\documentclass[12pt,reqno]{amsart}
\usepackage{amsmath, amssymb, amsthm}
\usepackage{mathtools}
\usepackage[margin=1in]{geometry}
\usepackage{color}

\newcommand{\ord}{\mathrm{ord}}

\newcommand{\wt}{\widetilde}
\newcommand{\blue}{\color{blue}}

\newcommand{\dd}{\mathrm{d}}
\newcommand{\edit}[1]{{\color{red}{$\clubsuit$#1$\clubsuit$}}}
\newcommand{\Ss}{{\mathbb S}}
\newcommand{\bcs}{\begin{cases}}
\newcommand{\ecs}{\end{cases}}

\newtheorem{theorem}{Theorem}[section]

\newtheorem{proposition}[theorem]{Proposition}
\newtheorem{corollary}[theorem]{Corollary}
\newtheorem{lemma}[theorem]{Lemma}
\newtheorem{remark}[theorem]{Remark}
\newtheorem{definition}[theorem]{Definition}

\numberwithin{equation}{section}

\newcommand{\Q}{{\mathbb Q}}
\newcommand{\Z}{{\mathbb Z}}

\newcommand{\acal}{{\mathcal A}}
\newcommand{\hcal}{{\mathcal H}}
\newcommand{\ical}{{\mathcal I}}

\numberwithin{equation}{section}

\newtheorem{theo}{{\sc Theorem}}

\newenvironment{rem}{\medskip\noindent{\it Remark:\/} }{\medskip}

\newcommand{\R}{{\mathbb R}}

\newtheorem*{main-theorem}{Main Theorem}
\newtheorem*{old-thm}{Theorem}
\theoremstyle{definition}
\numberwithin{equation}{section}

\newcommand{\fcal}{{\mathcal F}}
\def\11{\mathds{1}}

\def\Re{\,\mathrm{Re}\,}
\def\Im{\,\mathrm{Im}\,}

\def\supp{\mathrm{supp}\,}

\def\SS{{\mathbb S}}

\def\phi{\varphi}
\def\half{{\frac{1}{2}}}

\def\be{\begin{eqnarray*}}
\def\ee{\end{eqnarray*}}
\def\ben{\begin{eqnarray}}
\def\een{\end{eqnarray}}
\def\ker{\text{ker}}

\def\L2R{L_{\text{Rest}}^2}

\newcommand{\N}{{\mathbb N}}
\newcommand{\ccal}{\mathcal{C}}

\newcommand{\ecal}{\mathcal{E}}
\newcommand{\gcal}{\mathcal{G}}

\newcommand{\pcal}{\mathcal{P}}

\newcommand{\ocal}{\mathcal{O}}

\newcommand{\scal}{\mathcal{S}}

\newcommand{\Char}{\operatorname{Char}}

\begin{document}
\title[]{Restriction of eigenfunctions to totally geodesic submanifolds}

\author{Steve Zelditch}
\address{Department of Mathematics, Northwestern University, Chicago IL}
\email{s-zelditch@northwestern.edu}

\thanks{Research partially supported by NSF grant DMS-1810747}

\begin{abstract}  We prove a number of results on the Fourier coefficients $\langle \gamma_H \phi_j, e_k \rangle_{L^2(H)}$  of restrictions $\gamma_H \phi_j$ of Laplace eigenfunctions $\phi_j$ of eigenvalue $-\lambda_j^2$
of a compact Riemannian manifold $(M,g)$ of dimension $n$   relative to the eigenfunctions $\{e_k\}$ of eigenvalues $-\mu_k^2$ of a totally geodesic submanifold $H$ of dimension $d$. The results pertain to the  `edge case' $c=1$ where $|\mu_k - \lambda_j| \leq \epsilon$ for some $\epsilon  >0$   of  Kuznecov-Weyl sums $$N^1_{\epsilon, H}(\lambda)  =  \sum_{j,  \lambda_j \leq \lambda}  \sum_{k: | \mu_k -   \lambda_j| \leq \epsilon}   \left| \int_{H} \phi_j \overline{e_k}dV_H \right|^2.$$  We prove a universal  asymptotic formula  $N_{\epsilon, H}^{1} (\lambda) \sim
C_{n,d} \;\; a_{I} ^0(H, \epsilon) \lambda^{\frac{n+d}{2}   } $,  together with  universal estimates on the remainder and on jumps in   $N_{\epsilon, H}^{1} (\lambda)$. The growth of the Kuznecov-Weyl sums depends on $d = \dim H$, in contrast to    the  ``bulk cases" where $ | \mu_k - c   \lambda_j| \leq \epsilon, 0 < c < 1$, where the order of growth is $\lambda^{n-1}$ for submanifolds of any dimension (as 
shown  by Y. Xi, E. Wyman and the author).

\end{abstract} 

\maketitle

\tableofcontents

\section{Introduction}

Let   $(M, g)$ be  an $n$-dimensional  compact Riemannian manifold without boundary, let $\Delta_M = \Delta_g$ denote its  Laplacian, and let $\{\phi_j\}_{j=1}^{\infty} $ be an orthonormal basis of its eigenfunctions,
$$(\Delta_M + \lambda_j^2) \phi_j  = 0, \qquad \int_M \phi_j \overline{\phi_k} dV_M = \delta_{jk}, $$ 
where $dV_M$ is the volume form of $g$, and where the eigenvalues are enumerated in increasing order, 
$\lambda_0 =0 < \lambda_1 \leq \lambda_2 \cdots \uparrow \infty.$
Let  $H \subset M$ be a $d$-dimensional  embedded totally geodesic  submanifold with induced metric $g |_{H}$ with
volume form $dV_H$, let $\Delta_H$ denote the Laplacian of $(H, g |_H)$, and let
 $\{e_k\}_{k=1}^{\infty} $ be an orthonormal basis of its eigenfunctions on $H$,
$$(\Delta_H + \mu_k^2) e_k = 0, \qquad \int_H e_k \overline{e_j} dV_H = \delta_{jk}. $$ 
The purpose of this article is to study the  Fourier coefficients in the  eigenfunction expansion of the restriction
 $\gamma_H \phi_j = \phi_j |_H$ of $\phi_j$ to $H$,
\begin{equation} \label{FC}
\gamma_H \phi_j (y) = \sum_{k=1}^{\infty} \langle \gamma_H \phi_j, e_k \rangle_H\; e_k(y), \;\;  \langle \gamma_H \phi_j, e_k \rangle_{L^2(H)} : = \int_H \phi_j(y) \overline{e}_k(y) dV_H,
\end{equation}  
in the `edge case' where $\mu_k \simeq \lambda_j$ in the sense made precise below. 

Ideally, we would like to obtain asymptotics or estimates on the Fourier coefficients of individual eigenfunctions.  In special cases this is possible, e.g.
for closed geodesics on certain surfaces. But in general the Fourier coefficients may vary erratically as $\lambda_j$  varies or as $\mu_k$ varies, and 
to obtain asymptotics of Fourier coefficients  it is usually necessary  to study averages of squares of Fourier coefficients 
 in both the $\mu_k$ and $\lambda_j$ spectral parameters.  
We therefore study  thin window  or ``ladder Kuznecov sums'' in the sense of \cite{WXZ21},  \begin{equation} \label{cpsi}
  \left\{ \begin{array}{l}  
 N^{c} _{\psi, H  }(\lambda): = 
\sum_{j,  \lambda_j \leq \lambda}  \sum_{k=0}^{\infty}  \psi(c \lambda_j - \mu_k)  \left| \int_{H} \phi_j \overline{e_k}dV_H \right|^2,
 \\ \\
 N^{c} _{\epsilon, H  }(\lambda):= 
\sum_{j,  \lambda_j \leq \lambda}  \sum_{k: |c  \lambda_j - \mu_k| \leq \epsilon}   \left| \int_{H} \phi_j \overline{e_k}dV_H \right|^2, \end{array} \right.
 \end{equation}
 where  the test function $\psi \in \scal(\R)$ (Schwartz class).  It is shown in \cite{WXZ21} that the sums decay rapidly if $c > 1$. 
 The  asymptotics  for $c=0$ were first determined in \cite{Zel92} and those are those for $0 \leq c < 1$ are determined in 
 \cite{WXZ21} for any submanifold. In this article, and its sequel \cite{Z22}, the asymptotics are studied for the `edge case' $c=1$
 in the case where $H$ is a totally geodesic submanifold. It turns out that even the order of growth of the sums \eqref{cpsi}  are quite different from the case $c < 1$ and depend on $d = \dim H$. This is because  the edge Fourier coefficients  often have enhanced magnitudes compared to `bulk' Fourier coefficients with $c < 1$. The enhancement depends on the second fundamental
form of $H$, and is largest when $H$ is totally geodesic. The order of growth of the sums is smaller when $H$ has non-degenerate 
second fundamental form, and that case is not considered here.  We refer to Section \ref{COMPSECTINTRO} and Section \ref{COMPSECT} for a comparison of the 
cases $c < 1$ and $c=1$.

\begin{remark} The `ordering' $\lambda_j - \mu_k$ in \eqref{cpsi} and  \eqref{SpsiDEF} is important, and is adhered
to throughout, because it will
imply that asymptotically the argument of $\psi$ is positive. \end{remark}

 The Kuznecov sums \eqref{cpsi}  involve two types of localization: (i) localization of
 $\phi_j$ along $H$; and (ii) Fourier localization of $\phi_j |_H$ to a thin window of Fourier modes on $H$ of frequencies $\mu_j$ close to $\lambda_j$.  In \cite{WXZ21} the first sum in \eqref{cpsi} is called `smooth-sharp' since it involves the smoothed inner sum with $\psi$; the second sum is called `sharp-sharp' since it involves indicator functions in both
 $\lambda_j$ and in $\mu_k$.  The smoothed sum $ N^{1} _{\psi, H  }(\lambda)$ is technically simpler to work with
 and, perhaps surprisingly, often has better applications. Indeed, the inner sum of the sharp-sharp average can jump at certain $\epsilon$, as discussed in \cite{WXZ21}.  
  
 To state  first result we need some notation. By the injectivity radius $\rm{inj}(M,g)$ we mean the largest $R > 0$
 so that $\exp_x : B_x(R) \to M$ is a diffeomorphism to its image for all $x \in M$; $B_x(R) \subset T_x M$ denotes
 the ball of radius $r$ in the tangent space.  
 We say that the geodesic flow $G_X^t$ of any Riemannian manifold $(X,g)$ is `aperiodic'  if the set of closed geodesics of $(X, g)$  has Liouville measure zero in $S^*X$. We denote the volume form in geodesic coordinates $y = \exp_x \xi$ based at $x \in M$
 by $\Theta(x, y) dy. $ 
If we change
to geodesic  polar coordinates $(r, \omega)$, we get $dV_g = J(r, \omega) dr d \omega$ where \begin{equation} \label{VOLDEN} J(r, \omega) = r^{n-1} \Theta(r, \omega, y)
=  ||V_1(r) \wedge \cdots \wedge V_{n-1}(r) \wedge \frac{\partial}{\partial r} || \end{equation}
where $\Theta(r, \omega)= \Theta(x, y)$ when $x = (r, \omega)$. 
Here, $V_j$ is a basis of vertical Jacobi fields along the geodesic with direction $\omega$ and initial point $x$, i.e. Jacobi
fields satisfying $V_j(0) = 0$ and $\frac{DV_j}{D t}(0)$ is an orthonormal basis of the normal space to the geodesic. Also,
$\frac{\partial}{\partial r} $ is the unit tangent vector to the geodesic. For  example, on the standard sphere $\Ss^n$,  $J(r, \omega) =\sin^{n-1} r$.

 We assume throughout that $d = \dim H \geq 1$. If $d= 0$ and $H = \{x_0\}$ is a single point, then 
the   eigenfunctions of $H$ are constants, the only Fourier coefficient \eqref{FC} is the pointwise square $|\phi_j(x_0)|^2$, and
the Kuznecov-Weyl asymptotics of \eqref{cpsi}   reduce to pointwise Weyl law results (see \cite{DG75, HoIV, SV} for background).  The main result 
is a generalization of such pointwise Weyl laws to higher dimensional submanifolds.

\begin{theorem} \label{main 1}
Let $\dim M = n$ and let $\dim H =d \geq 1$.  Assume that $H$ is totally geodesic. Let $\hat{\psi} \in C_0^{\infty}(\R)$ be a real, positive, even test function supported   in the set $(-r_0, r_0)$ where  $r_0 < \rm{inj}(M,g)$. Then,  there exist universal constants $C_{n,d}$ such that for any $\epsilon > 0$,
$$N_{\psi, H}^{1} (\lambda) =
C_{n,d} \;\;   a^0_1(H, \psi) \; \lambda^{\frac{n+d}{2}   } + R_{\psi, H}^{1} (\lambda), \;\; \rm{where}\; R_{\psi, H}^{1} (\lambda) =   O(\lambda^{\frac{n+d}{2} -1}).$$
The leading coefficient is given by,
\begin{equation} \label{acHpsi1}  a^0_1(H, \psi): =  \; \int_{\R} \int_H \int_{S_q^*H}  \hat{\psi} (s) \;(s + i 0)^{- \frac{n-d}{2}} 
\Theta^{-\half }_M(q,\exp_q s \omega)  \Theta^{-\half}_H(q, \exp_q s \omega) ds dV_H(q) d S(\omega). \end{equation}

If the geodesic flow $G^t_H$ of $H$ is aperiodic, then 
\begin{equation} \label{APERIODIC} R_{\psi, H}^{1} (\lambda) =   o(\lambda^{\frac{n+d}{2} -1}). \end{equation}

\end{theorem}

In the sequel \cite{Z22}, we study two term asymptotics of \eqref{cpsi}, which give necessary conditions to obtain maximal growth of individual terms
(see Section \ref{2TERMINTRO} for further remarks).  The last statement for aperiodic flows is a two-term asymptotic in which the second term
of order $\lambda^{\frac{n+d}{2} -1}$ vanishes. The proof that it vanishes is the same as in \cite[Section 5.3]{WXZ21} but will be reviewed in 
Section \ref{SUBPRINCIPAL}.
Before stating  further results,  or discussing the proof of Theorem \ref{main 1}, we address some issues in both the assumptions and conclusions
of the theorem  which explain the organization and the length  of this article. 

Note that when $d = n-1$, the density $(s + i 0)^{- \frac{n-d}{2}} $ is integrable and there is no need for regularization. When $d < n-1$,
the asymptotics of Theorem \ref{main 1} 
are somewhat subtle  because the  contribution of very short (wave-length scale) distances in $s$ is  larger than the asymptotics in Theorem \ref{main 1},
i.e. substantial cancellation is required in the integral over $\rm{supp} \hat{ \psi}$  to produce the relatively small order of growth $\lambda^{\frac{n+d}{2}}$.
The unusual leading coefficient \eqref{acHpsi1} signals the existence of a `blow-down' singularity in the
oscillatory integrals used in the proof of Theorem \ref{main 1}. The geometric origin of the blow-down singularity
is simple to understand:  it is due to the collapse of distance spheres on $H$ when $s=0$ (see Section \ref{BIANGLESECT} and Section \ref{BDSECT}). This singularity does not occur for $c < 1$ in \cite{WXZ21}. It indicates that 
the asymptotics in Theorem \ref{main 1} cannot be deduced purely from the Fourier integral operator theory
under clean compositions, or equivalently, by a stationary phase analysis. To prove the formula for the leading term, we use stationary phase outside
of a small interval around $s=0$ and then reduce to a universal model integral for the short distance part (Section \ref{MODELSECT}). 
It is    proved in Section \ref{FULLAMP} that the procedure used in the model case extends to give the result for the general
case, completing the proof of the leading term and remainder in Theorem \ref{main 2} but without calculating the full amplitude; that is done
in Section \ref{HADPAR}.
 It is simple to obtain the coefficient under the (rather artificial)  assumption that $\hat{\psi} =0$ in some interval
$[-\epsilon, \epsilon]$ around $s=0$, but not so simple to determine the regularization or to show that there does
not exist another term supported at $s =0$, as for $c < 1$ in \cite{WXZ21}.  
The evaluation of the coefficient \eqref{acHpsi1} is given in   Section \ref{SINGINTRO} and is indirect: Once the existence of the
asymptotic expansion is proved, it  is fairly obvious that $\psi \to a_1^0(H, \psi)$ 
is a positive distribution, hence a positive measure. Due to  the ordering $\lambda_j - \mu_k \geq 0$ it is supported in $\R_+$. 
As a result,  the  Fourier transform of the distribution  \eqref{acHpsi1}  has
a holomorphic extension to the upper half plane, as indicated in \eqref{acHpsi1}.

The most computable examples corroborating the leading order term and proving sharpness of the remainder are that of totally geodesic 
spheres $\Ss^d \subset \Ss^n$ of standard spheres,   subspaces of $\R^n$ (Section \ref{BESSELSECT}) and  non-degenerate closed geodesics  of any Riemannian manifold $(M,g)$. In \cite{Z22},  the asymptotics are derived by Gaussian beam techniques.

A second  aspect of the asymptotics is due to the dependence on the test function $\psi$. In special cases, again such as
 $\Ss^d \subset \Ss^n$, the eigenvalue differences $\lambda_j - \mu_k$ are essentially differences $N - M$ of positive integers. This raises
the general  question (for any $c \in (0,1])$  of when the ordered difference spectrum, \begin{equation} \label{SigmaHM} \Sigma_{M,H}(c) = \{c\lambda_j - \mu_k\}\end{equation}
is dense in $\R$.
  If instead
of assuming that $\hat{\psi} \in C_0^{\infty}(\R)$ we assumed that $\psi \in C_0^{\infty}(\R)$, we could take $\psi$ supported in a short
interval $[-\half, \half]$ which is a gap in the difference spectrum $N - M$, and for such $\psi$ \eqref{cpsi} would equal zero. Thus,
we see that there are different types of asymptotics problems, some involving $\hat{\psi}$ with small support
around $0$, including wave-length support, and some involving the opposite regime where $\psi$ has small support,
and that different techniques are needed for the different regimes.  The question of when  \eqref{SigmaHM} is dense is reminiscent of the 
Helton clustering theorem; it appears that the remainder term $R^1_{\psi, H}$ is maximal only when the  difference spectrum \eqref{SigmaHM} fails to be dense. We leave these questions to future work.

In this article we assume
 that $\hat{\psi} \in C_0^{\infty}$ and $\rm{Supp} \hat{\psi} \subset (-r_0, r_0)$. 
The assumption that  $\hat{\psi} \in C_0^{\infty}$ often arises in a Fourier integral analysis of spectral problems to 
limit the number of singularities of the dual sum,
\begin{equation}\label{SpsiDEF}
\begin{array}{lll}
S(t, \psi): & = &  \sum_{j,k} e^{it \lambda_j } \psi (\lambda_j - \mu_k)  \left| \int_{H} \phi_j \overline{e_k}dV_H \right|^2.  \end{array}
\end{equation}  By studying the  wave equations of $(M,g)$ and of $(H, g |_H)$, one finds that for 
$\hat{\psi} \in C_0^{\infty}(\R)$ the leading order term \eqref{acHpsi1} involves the times $t$ of periodic orbits of the geodesic flow
$G_M^t$ of $(M,g)$ and all times $s$, including  the times $s$ of periodic orbits, of the geodesic flow of $G_H^s$ of $(H, g|_H)$ (see Section \ref{WFSECT}). 
It is evident from the factor $\Theta^{-\half }_M(q,\exp_q s \omega)  \Theta^{-\half}_H(q, \exp_q s \omega)$ in \eqref{acHpsi1} that the distribution  $ a^0_1(H, \psi) $  for general $\hat{\psi} \in C_0^{\infty}$ depends on the structure of periodic orbits and conjugate points.  The formula
\eqref{acHpsi1}  uses the assumption that $\hat{\psi}$ is supported in the interval $(-r_0, r_0)$, since the volume densities
$\Theta_M(q, q')$, resp. $\Theta_H(q,q')$ become singular when $q'$ is
a conjugate point of $q$. The formula suggests that the leading coefficient is valid for any $\hat{\psi} \in C_0^{\infty}$ if one suitably 
regularizes the integrand at conjugate points. For instance, in the case of a totally geodesic subsphere $\Ss^d \subset \Ss^n$, one has the
global in $s$ formula,
\begin{equation} \label{a1SPHERE}  a^0_1(\Ss^d, \psi) = \int_{\R} \hat{\psi}(s) (\sin (s + i0))^{- \frac{n-d}{2}}  ds. \end{equation} 

The additional assumption that $\rm{Supp} \; \hat{\psi}\; \subset (-r_0, r_0)$ ensures that there is only one singularity of \eqref{SpsiDEF}, namely at $t=0$, and it does not involve any periods or conjugate points of the geodesic
flow on $H$. This is already sufficient to prove one term asymptotics
with a sharp remainder.    If we allow $\psi \in C_0^{\infty}(\R)$, the Fourier integral operator techniques of
this article would require a  sum over all periods of the geodesic flows and the test function $\hat{\psi}(s)$ would have to have a compensating decay.   This is far from saying that the case of $\psi \in C_0^{\infty}(\R)$ is uninteresting. Such high localization in the difference 
spectrum is often of the highest interest in applications. But it explains why we do not consider that regime in this article.For further discussion of the singularity of \eqref{acHpsi1} we refer to Section \ref{SINGINTRO}.

Two final notation remarks.
First, the  constant $C_{n,d}$ is  computed from the Hadamard parametrix method
in \cite{Be} (see Section \ref{HADPAR}).  The formula for \eqref{acHpsi1}  requires determining the precise regularization
of the distribution $s^{- \frac{n-d}{2}}$ at $s=0$.  We often make use of the relations \cite[Page 93 and Page 172]{GS},

\begin{equation}\label{GS} \left\{\begin{array}{l} (s \pm i 0)^{- \frac{n-d}{2}} = s_+^{- \frac{n-d}{2}} + e^{\mp i \pi  \frac{n-d}{2}} s_-^{- \frac{n-d}{2}} = s_+^{- \frac{n-d}{2}} + (\mp i)^{n-d} s_-^{- \frac{n-d}{2}}  ,\\ \\ \int_0^{\infty} e^{i t \sigma} t_+^{\lambda} d t
= i e^{i \lambda \pi/2} \Gamma(\lambda + 1) (\sigma + i
0)^{-\lambda - 1} \end{array}  \right.
\end{equation}
which explain the origins of the regularization and of the normalizing constants $C_{n,d}$.   Second, we use the convention that the
propagator of an operator $P$  by $U(t) = e^{it P}$ (rather than $e^{- it P}$) and always integrate it against $e^{- it \lambda}$.

 \subsection{\label{INTROJUMPSECT} Jumps in the Kuznecov-Weyl sums}

Theorem \ref{main 1} concerns asymptotics of  a double average \eqref{cpsi}, an outer average over the eigenvalues $\lambda_j$ of $\sqrt{-\Delta}_M$ and an inner average over
the eigenvalues $\mu_k $ of $\sqrt{-\Delta}_H$.  
One may  obtain results on  eigenfunctions of individual eigenvalues to some extent by studying  the jumps of the  Kuznecov-Weyl sums \eqref{cpsi} at the eigenvalues $\lambda_j$,
weighted by either smooth test functions  $\psi$ or sharp indicator functions:
\begin{equation} \label{JDEFpsi}  \left\{ \begin{array}{l} J^{1} _{ \psi, H  }(\lambda_j): =   \sum_{\ell: \lambda_{\ell} = \lambda_j} 
\sum_{ k}  \psi( \lambda_j -  \mu_k)  \left| \int_{H} \phi_{\ell} \overline{e_k}dV_H \right|^2, \\ \\
 J^{1} _{ \epsilon, H  }(\lambda_j): = \sum_{\ell: \lambda_{\ell} = \lambda_j} 
\sum_{  \substack{k: | \lambda_j - \mu_k|  \leq \epsilon}}   \left| \int_{H} \phi_{\ell} \overline{e_k}dV_H \right|^2.

\end{array} \right.\end{equation}
The sum over $\ell$ is a sum over an orthonormal basis for the eigenspace $\hcal(\lambda_j)$ of $\Delta_M$ of eigenvalue
$\lambda_j$.  If the $M$-eigenvalues have multiplicity one, then there is a single term in the $\lambda_j$ sum.  On the other hand,
unless the $H$-eigenvalues come in clusters with high multiplicities and  are  separated by gaps, the $\mu_k$-sum runs over
roughly $\ocal(\lambda_j^{d-1})$ eigenvalues. We refer to the sums of  \eqref{JDEFpsi} as the `inner sums'. 

Theorem \ref{main 1} implies the following universal bounds on remainders and jumps.

 \begin{corollary} \label{theo JUMP}  With the same assumptions and notations as in Theorem \ref{main 1},  for
any postive test function $\psi$ with $\hat{\psi} \in C_0^{\infty}(\R)$ having small support,  there exists a constant $C_{\psi}  > 0$ such that $$
\begin{array}{ll} (i) &
 J_{\psi, H}^{1} (\lambda) \leq	C_{\psi} \; \lambda^{\frac{n + d}{2} -1} 
 . \end{array}
  $$
  Moreover, for any $\epsilon > 0$ there exists $C_{\epsilon} > 0$ so that
 $$ \begin{array}{ll} (ii) &
 J^{1}_{\epsilon, H}(\lambda_j) \leq	C_{\epsilon} \; \lambda^{\frac{n + d}{2} -1}
 . \end{array}  $$
 In the aperiodic case, 
 \begin{equation} \label{JUMPAP} J_{\psi, H}^c(\lambda) = o(\lambda^{\frac{n + d}{2} -1}). \end{equation}
\end{corollary}

The first statement  follows from the  standard observation that,\begin{equation} \label{epjump} \begin{array}{lll} J_{\psi, H}^1 (\lambda_j) & = & 
N^1_{\psi, H} (\lambda_j) - N^1_{\psi, H} (\lambda_j - 0) \\&&\\& = & 
R^1_{\psi, H} (\lambda_j) - R^1_{\psi, H} (\lambda_j - 0),   \end{array} \end{equation}
since the leading order term in Theorem \ref{main 1} is continuous. The  remainder bound (i)  thus follows from  Theorem \ref{main 1}. To 
prove (ii) we use (i) and an appropriate choice of $\psi$ (see Section \ref{PFTHEOJUMP}).

 To obtain sharper results, it is necessary to study the long time behavior of geodesic
flows and wave groups of $(M,g)$ and of $(H, g |_H)$.  This is carried out in the sequel \cite{Z22}. 
 The proof of Theorem \ref{main 1} and of the other results of this article mainly  use the `small time' behavior of the geodesic flows $G^t_M$ of $(M, g)$, resp. $G^s_H$ of $(H, g |_H)$. An exception is the statement
\eqref{APERIODIC}, which follows from the wave front set analysis in Section \ref{WFSECT}, analogous
to the wave front analysis in \cite{WXZ21}. In the short time analysis we employ both the H\"ormander parametrix (Section \ref{HORPAR}) and the Hadamard parametrix method (Section \ref{HADPAR}).

\subsection{\label{COMPSECTINTRO}  Comparison to $c < 1$}
Let us compare Theorem \ref{main 1} to the analogous result  \cite[Theorem 1.1]{WXZ21} in the case  $c < 1$. In that case,  \begin{equation}
\label{c<1} N_{\psi, H}^{c} (\lambda) =
 C_{n,d} \;\; a_c^0(H, \psi) \lambda^{n-1 } +   O(\lambda^{n-2}), 
 \end{equation}
where the leading coefficient is given by the temperate  distribution,
$$ a_c^0(H, \psi): = 
 \hat{\psi}(0) \; c^{d-1} (1 - c^2)^{\frac{n-d-2}{2}}  \hcal^{d}(H)$$
 if $\hat{\psi}$ is supported in a sufficiently small interval around  $s=0$. 
Note that  the power of $\lambda$  in  the $c < 1$ case  is independent of $d = \dim H$. To bridge the two results, note that 
if $c = 1 - \frac{\epsilon}{\lambda}$ then
$  (1 - c)^{\frac{n-d-2}{2}}  =  (\frac{\epsilon}{\lambda})^{\frac{n-d-2}{2}}  = r(\epsilon) \lambda^{- ^{\frac{n-d-2}{2}}}, $
and 
$  (1 - c)^{\frac{n-d-2}{2}} \lambda^{n-1} = r(\epsilon) \lambda^{- ^{\frac{n-d-2}{2}}}  \lambda^{n-1} = r(\epsilon)  \lambda^{\frac{n+d}{2}}. $

If $\hat{\psi} = 0$ near $s =0$, then the order of magnitude of $N_{\psi, H}^c(\lambda)$ drops by $1$ when $c < 1$, but not when $c =1$. A further result,  \cite[Theorem 1.18]{WXZ21} in the case 
$c < 1$ shows that the leading coefficient increases as the support of $\hat{\psi}$ increases (assuming $\hat{\psi} \geq 0)$, with jumps at special
values $s_j$. In contrast, Theorem \ref{main 1} shows that when $c=1$, the leading coefficient is continuous in $s$ and increases with the 
support of $\hat{\psi}$ (again, assuming $\hat{\psi} \geq 0$). Further comparisons are given in Section \ref{GB0} and in  Section 
\ref{COMPSECT}.

We also see that the order of growth $\lambda^{\frac{n+d}{2}}$ in the case $c=1$ is greater
than the order of growth $\lambda^{n-1}$ for $c< 1$ if and only if $d= n -1$ (the hypersurface case),  that the two orders are equal
when $d = n-2$ and that the asymptotics  for $c < 1$ are of higher order than for $c=1$ if $d \leq n-3$.

The very different powers in Theorem \ref{main 1} and for  \eqref{c<1} is due to the different type of contributions of conormal directions to $H$.
It is explained below in Section \ref{BIANGLESECT} that the powers of $\lambda$ are controlled by the dimension of the set of initial
vectors $\xi \in S_H^*M$ of geodesic bi-angles. The parameter $c$ is defined by $c = \frac{|\pi_H \xi|}{|\xi|}$ where $\pi_H: T_q^* M \to T_q^*H$ is the
orthogonal projection at $q \in H$. We denote  by $S^c_H M$ the covectors this equality with footpoint on $H$, and we denote the decomposition into tangent and normal components by   $\xi = \xi_T + \xi^{\perp}$. When $c < 1$,  $|\xi^{\perp} |^2 = \sqrt{1 - |\xi_T|^2} 
= \sqrt{1 - c^2}$ and there the set of conormal parts of unit vectors  has dimension $n-d-2$.   When $c=1$ the co-normal part  $\xi^{\perp}$ must vanish, the dimension drops by $n-d-2$, and the power drops by $\frac{n-d-2}{2}$.  As mentioned above, the singularity  is of   the  ``blow-down''  type,  exhibited by integrals of Bessel functions (cf. Section \ref{BESSELSECT}).  
 We refer to Lemma \ref{ORDERLEM} for the modified order calculation following the methods of \cite{WXZ21}, and to Section
  \ref{BDSECT} for discussion of the blow-down singularity.

\subsection{\label{BIANGLESECT}  Geodesic geometry}  Before sketching the proof of Theorem \ref{main 1} and before stating additional results, we discuss the  geodesic geometry that is responsible for the singularity at $s=0$ of the  coefficient \eqref{acHpsi1}, and the geodesic geometry which is responsible for maximal jumps in Corollary 
\ref{theo JUMP}. We will prove the statements on maximal jumps in \cite{Z22}.

In the case $c < 1$ of  \cite{WXZ21}, the relevant dynamics was given by geodesic bi-angle geometry.    For any $c$, the singularities of \eqref{SpsiDEF}  are governed by the dynamical equation,
 \begin{equation}\label{GEO1} G_H^{-s} \pi_H G_M^{t + c s} (q, \xi) = \pi_H(q, \xi), \;\; q \in H, \xi \in S^c H. \end{equation}
 where $G_M^t$ (resp. $G_H^s$) is the geodesic flow of $(M,g)$ resp. $(H, g|_{TH})$), where  $\pi_H: S_q^* M \to S_q^*H$ is the orthogonal projection, where $c = \frac{|\pi_H \xi|}{|\xi|}$ and where $S_H^c M $ is the set of covectors $\xi \in T^*_H M$ satisfying this constraint.  To be precise, \eqref{GEO1} is the equation that holds on the critical set (or wave front relation)
 of the governing Fourier integral operator. We think of the condition $\xi \in S^c_H M$ as a `constraint' on the shape of the bi-angles.  For $c < 1$, the solutions $(t, s, q, \xi)$ correspond to geodesic bi-angles starting at $q$ and ending at $\exp_q s  \pi_H \xi$,  consisting of one  leg given by an $M$-geodesic of  length $s + t$, one leg consisting of an $H$-geodesic of length $s$, and with compatible initial and terminal velocity vectors. When $c=1$, and $H$ is totally geodesic, $\xi \in  S^*H$ and $G_H^s(q, \xi) = G_M^s(q, \xi)$, so  the equation
\eqref{GEO1} simplifies   to 
  \begin{equation} \label{GEO2} G_H^{t}  (q, \xi) = (q, \xi), \;\; q \in H, \xi \in S^*H.  \end{equation}
The equation is independent of $s$, explaining why \eqref{acHpsi1}  is an integral over all $s$.  It also indicates that an important geodesic condition is periodicity  of the geodesic flow is that of $H$. In general, we say that the geodesic flow
of a Riemannian manifold $(X, g)$ is periodic if there exists $T \not=0$ such that $G^T_X = Id$. 

The  simplest example where  maximal jumps occur
is the case of totally geodesic subspheres $\Ss^d \subset \Ss^n$ of spheres, where both $G^s_H$ and  $G^t_M$ are periodic. Examples where maximal jumps do not occur
are  Riemannian products $\SS^d \times \Ss^{n-d}$, despite periodicity of $G_H^s$ (see Section \ref{MODSECT}). This  raises the question whether periodicity of $G_M^t$ is also necessary for maximal jumps. In fact, it is not necessary. An example is where $M$ is a convex surface of
revolution and $H$ is its unique rotationally invariant geodesic. The conditions for maximal jumps to occur will be given
in \cite{Z22}. Some indications of what is involved are given in Section \ref{FURTHER}.

To tie the discussion of the equation \eqref{GEO1} together with that in \cite{WXZ21}, we recall the following
standard definition from \cite{DG75}.
\begin{definition} \label{CLEAN}
We say that the set $\gcal_c$,  resp. $\gcal_c^0$ of solutions of \eqref{GEO1}, resp. \eqref{GEO2}, is {\it clean} if $\gcal_c$, resp. $\gcal_c^0$, is a submanifold of $\R \times \R\times S^c_H M$, resp. $\R \times S^c_H M$, and if its tangent space at each point is the subspace fixed by $D_{\zeta}  G_H^{ - s} \circ \pi_H \circ G_M^{cs + t} $ (resp. the same with $s = 0$), where $\zeta$ denotes a  point of $S^c_H M$.
\end{definition}
  
The equation \eqref{GEO2} is simply the equation for periodic geodesics of $G^t_H$ and cleanliness is the condition
that for each period $T$, the set of closed geodesics of period $T$ is a submanifold of $S^*M$ whose tangent space
is the fixed point set of $DG^t_H$ (see \cite{DG75} for this case). Even when this occurs, the Fourier integral 
operator compositions of Section \ref{WFSECT} fail to be clean when $c =1$ and $s=0$ essentially due to the collapse of the fibers of the co-normal sphere bundle, as discussed in Section \ref{COMPSECTINTRO} (see also Section
\ref{COMPSECT}). Since this failure is universal, it does not depend on the metrics or geodesic flows. 

\subsection{Outline of the proof of Theorems \ref{main 1} and further results }

To prove Theorem \ref{main 1} we first study a doubly-smoothed version, 
\begin{equation} \label{NpsirhoDEF} N^{1} _{\psi, \rho, H  }(\lambda) = \sum_{j, k=0}^{\infty} \rho(\lambda - \lambda_j)  \psi(\lambda_j - \mu_k)  \left| \int_{H} \phi_j \overline{e_k}dV_H \right|^2, \end{equation}
with a second test function $\rho \in \scal(\R)$ with $\hat{\rho} \in C_0^{\infty}$.  To obtain the asymptotics of \eqref{NpsirhoDEF},
we study the singularities  of the  Fourier transform, $$N^{1} _{\psi, \rho, H  }(\lambda) = \int_{\R} \hat{\rho}(t) e^{it \lambda} S(t, \psi)dt, \; $$ 
where $S(t, \psi)$ is defined in \eqref{SpsiDEF}.  In terms of wave kernels, we define
\begin{equation} \label{S(s,t)DEF} S(s, t) = \int_H \int_H \gamma_H U_M\gamma_H^* (t  + s, q, q') U_H(-s, q, q') dV_H(q) dV_H(q'), \end{equation}
and then, 
\begin{equation} \label{S(t, psi)DEF} S(t, \psi) = \int_{\R} \hat{\psi}(s) S(s, t) ds. \end{equation}

In Theorem \ref{main 1} we assumed that the only period of $G^s_H$ in the support of $\hat{\psi}$ is $s= 0$. In studying 
\eqref{NpsirhoDEF}, we further need assumptions of the periods of $G_M^t$ in the support of $\hat{\rho}$.
We study the short time  singularities of $S(t, \psi)$ using a reduction to a model phase,  unlike the methods employed in \cite{WXZ21}. 
The model phase is most visible if one uses the H\"ormander type parametrix, whose phase is linear in $t$ (Section \ref{HORPAR}).
In Section \ref{HADPAR} we use a Hadamard parametrix to calculate the amplitude.   

To prove the `aperiodicity' result of Theorem \ref{main 1} we also need some input from long time
singularities from Section \ref{WFSECT}. To this end,
we repeat two  definitions from \cite{WXZ21} but point out that they simplify
when $c=1$ and when $H$ is totally geodesic. The following is the $c=1$ analogue of \cite[Definition 1.15]{WXZ21}.

 \begin{definition} \label{SOJOURNDEF} Let  $$\Sigma^1 := \{t: \rm{Fix} (G^t_H) \not= \emptyset\} $$
 be the set of periods of closed geodesics of $H$.

\end{definition}

In \cite[Definition 1.15]{WXZ21}, the analogous set $\Sigma^c(\psi ) $ for $c < 1$ was defined by 
 as  the set of singular points $t$  of $S(t,\psi)$; it  consists of 
 of $t$ for which there exist solutions of \eqref{GEO1} with $s \in \supp \hat{\psi}.$  When   $c=1$ and $H$ is totally geodesic, it follows from   \eqref{GEO2} that $\Sigma^1(\psi)$ is independent of $\psi$. 
 The order of
the singularity at $t$  corresponds to the dimension of the solution set of \eqref{GEO2}.

 \begin{definition} \label{DOMDEF}
 We say that the singularity at $t=0$ is {\it dominant} if the dimension of the solution set  \eqref{GEO1}   at $t=0$  is strictly 
  greater than any other $t \in \Sigma^1$. \end{definition}
  By the remarks above, the singularity at $t=0$ is dominant except when $G^t_H$ has a positive measure
  of periodic orbits. In the clean case (cf. Definition \ref{CLEAN}), this would mean that $G^t_H$ is periodic, i.e. that
  $(H, g_H)$ is a Zoll manifold.  
  
Although $\Sigma^1$ is independent of $\psi$, the periods $s \in \Sigma^1$ in the support
  of $\hat{\psi}$ play an important role in the coefficient \eqref{acHpsi1} and indeed cause the difficulties which 
  led to the assumption in Theorem \ref{main 1} that $\rm{Supp} \hat{\psi} \subset (-r_0, r_0)$. These periods do
  not affect the order of the singularities at $t \in \Sigma^1$ (including $t=0$) but they do affect the coefficient
  generalizing \eqref{acHpsi1}.

  The following theorem is analogous
  to Theorem \ref{main 1} but employs a smooth test function $\rho$ instead of the sharp interval sum. Its assumptions
  qualify it as a `short-time' theorem.


\begin{theorem} \label{main 2} Let $\dim M = n, \dim H = d$, and assume $H$ is totally geodesic.  Let $\psi, \rho \in \scal(\R)$ with $\hat{\psi}, \hat{\rho} \in C_0^{\infty}(\R)$.  Assume that
$\rm{supp} \;\hat{\rho} \cap \Sigma^1= \{0\}$ and  $\hat{\rho}(0) =1$.   Let $\hat{\psi} \in C_0^{\infty}(\R)$ be a real, even test function supported   in the set $(-r_0, r_0)$ where  $r_0 < \rm{inj}(M,g)$.   Then, there exists a complete asymptotic expansion of $N^{1} _{\rho, \psi, H  }(\lambda)$ of the form, 
$$  N^{1} _{\rho, \psi, H  }(\lambda) \sim    \lambda^{\frac{n +d}{2}-1}   \sum_{j=0}^{\infty} \beta_j \; \lambda^{-j}, $$
with  $\beta_0 = a_1^0(H, \psi)$ \eqref{acHpsi1}, and
where $C_{n,d}$ is a universal dimensional constant.

\end{theorem}

This theorem seems similar to \cite[Proposition 2.1]{DG75} and many similar asymptotics results that are
proved under clean composition hypotheses. Equivalently, they are proved by the stationary phase method.  As mentioned above, the compositions relevant to Theorem \ref{main 2}
are not clean due to a blow-down singularity at $s=0$ and cannot be proved solely by stationary phase methods. 
Rather we take a hybrid approach in Section \ref{MODELSECT} that combines stationary phase in certain variables
and a direct integration in others.  The order of the asymptotics is consistent with a formal application of stationary phase, but it is not clear apriori that this formal order is correct because the blow-down singularity could dominate the asymptotics.  We emphasize this point with the Euclidean case (Bessel integrals) in Section \ref{BESSELSECT}.

In Proposition \ref{MORESINGS} we extend the proof of Theorems \ref{main 1}- Theorem \ref{main 2} to test functions $\rho$ with 
arbitrary support but retain the assumption that $\rm{Supp} \hat{\psi} \subset (-r_0, r_0)$. The extension then does
not involve any new inputs beyond the wave front analysis in Section \ref{WFSECT}.

Once Theorem \ref{main 2} is proved, the  `sharp' Weyl asymptotics of Theorem \ref{main 1} (except for the last
statement about aperiodic flows)  follow from Theorem \ref{main 2} by a standard cosine Tauberian theorem (Section \ref{SS Tauberian}).  We discuss the proof of the aperiodic flow in the next subsection.

To prove Theorem \ref{main 2}, 
we  construct parametrices  for \eqref{NpsirhoDEF} when $\rm{Supp} \hat{\rho} $ lies in a sufficiently
short interval around $t=0$ so that no other singularities of \eqref{SpsiDEF}  lie in the interval. We use the parametrix to calculate the leading
order contribution in Theorem \ref{main 1}. 
Due to the fact that the clean composition calculus breaks down at small distances, we rely on a small time parametrices when studying small distances. 
When the geodesic flow of $(M, g)$ is periodic, the small distance problem will recur for larger times as well; but the main term of the  asymptotics is due to the
$t =0$ singularity and that can be analysed using the small time parametrix.

\subsection{The remainder estimate in the case where $G_H^s$ is aperiodic} The last statement of Theorem \ref{main 1} says that the remainder is `small oh'' of the universal remainder when $G_H^s$ is aperiodic. This requires inputs from the long time behavior of the wave group and geodesic flow. But the inputs are less than those necessary to generalize Theorem \ref{main 1} to general $\hat{\rho}, \hat{\psi} \in C_0^{\infty}$.  We retain the assumptions on $\hat{\psi}$ but relax the assumption on $\hat{\rho}$ and let its support be arbitrarily large, i.e. we consider all singularities
of $S(t, \psi)$ where  $\rm{Supp} \hat{\psi}
\subset (-r_0, r_0)$. 


\subsection{Dependence of $N^1_{\epsilon, H}$ and of the jumps $J^1_{\epsilon, H}(\lambda_j) $ on $\epsilon$}\label{epsilonSECT} 
It may seem more interesting to analyse the sharp-sharp sums  $N^1_{\epsilon, H}$ of \eqref{cpsi}. 
However, they have an erratic dependence on $\epsilon$, just as $N^1_{\psi, H}(\lambda)$ has a complicated dependence
on $\psi$ if we drop the support assumptions in Theorem \ref{main 1}. The situation is similar to that for $c < 1$, discussed in detail in   \cite[Section 1.10]{WXZ21} (see \cite[Theorem 1.24]{WXZ21}. The  asymptotics of the sharp-sharp sum $N^c_{\epsilon, H}(\lambda)$
for $0 < c< 1$ 
are proved to  have the form, $$N^c_{\epsilon, H}(\lambda) = a_c^0(H, \epsilon) \lambda^{n-1} + o(\lambda^{n-1}), $$ 
where the leading coefficient $a_c^0(H, \epsilon)$ is linear in $\epsilon$. The surprisingly large remainder estimate is sharp  on spheres, due to
jumps in the jump $J^c_{\epsilon, H} $ at special values of $\epsilon$. Note that $\hat{\psi}(0) = \int \psi =  2 \epsilon$ when $\psi = {\bf 1}_{ [-\epsilon, \epsilon]}$. This case is similar to (but more irregular than) the case where $\psi \in C_0^{\infty}$.

\subsection{\label{FURTHER} Background, related results and problems}

Fourier coefficients of restrictions of eigenfunctions are central to the theory of automorphic forms. They were studied classically by Hermite and Jacobi for Fourier coefficients of modular forms  around closed horocycles for the
modular surface ${\mathbb H}^2/SL(2, \Z)$, and later by H. Petersson \cite{P32}  around closed geodesics. The study of $C^{\infty}$ eigenfunctions  of the Laplacian on a hyperbolic surface was developed by H. Masss and was studied
systematically by N.V. Kuznecov \cite{K80} using the Kuznecov sum  formula. Since then, the study of Fourier coefficients around various submanifolds has formed an important part of automorphic forms on very general arithmetic  locally symmetric manifolds. By far the sharpest estimates on individual Fourier coefficients
are those of  \cite{ M16} for geodesic Fourier coefficients of Hecke eigenfunctions.

Averages of Fourier coefficients have been studied in many articles since Kuznecov's article \cite{K80}.  In the automorphic forms literature, the weights or test functions used in Kuznecov's formula are adapted to the setting of hyperbolic quotients or other locally symmetric manifolds. In this article, as 
 in \cite{Zel92, WXZ20, WXZ21}, we use the wave equation and associated  test functions. As often happens,  wave equation methods give sharper remainder terms than other methods, and this is true in the present applications.

 A distant goal is to obtain asymptotics of  ``empirical measure of the Fourier coefficients"  \eqref{FC}  of an individual eigenfunction $\phi_j$ as $\mu_k$ varies, i.e. the measure whose mass at $\mu_k$ is the modulus square $|\int_H \phi_j e_k|^2$
of the Fourier coefficient.  The results of this article and of \cite{WXZ21} concern the integrals of the empirical measure over short intervals.

 We have already discussed the long-time refinement of the results of Theorem \ref{main 1} above; they are studied 
 in \cite{Z22}. In addition, we point out the following closely related problems.
 
 \subsubsection{\label{SIMPLE} Simple generalizations}

 With no additional effort, all results  of this article extend to the more general Kuznecov-Weyl sums, 
for any $c \in [0,1]$, and for any $f \in C^{\infty}(H)$,
\begin{equation} \label{cpsif}
 N^{c} _{\psi, H, f }(\lambda): = 
\sum_{j,  \lambda_j \leq \lambda}  \sum_{k=0}^{\infty}  \psi( \mu_k -  c \lambda_j)  \left| \int_{H} f  \phi_j \overline{e_k}dV_H \right|^2.
 \end{equation}
In the  leading coefficient \eqref{acHpsi1}, for $c=1$,  $\hcal^d(H)$ is replaced by
$  (\int_H f d V_H)$. More generally, we could replace $f$ by a semi-classical pseudo-differential operator  $Op_H(a)$ along $H$ and 
replace the inner products by $\langle Op_H(a) \gamma_H \phi_j, e_k \rangle_{L^2(H)}$ and then the coefficient is $\int_{S^*H} a_0 d\mu_H$
where $a_0$ is the principal symbol and $d\mu_H$ is the Liouville measure on $S^*H $. We will use the generalization
to $Op_H(a)$ in Section \ref{Aperiodic}.

\subsubsection{\label{LOG} More refined remainder estimates}

The Hadamard parametrix of Section \ref{HADPAR} can be used  on manifolds without conjugate points  to obtain a global-in-time
parametrix for the half-wave group  \cite{Be}. Such manifolds always have aperiodic geodesic flows. 
By controlling the exponential growth rate of Jacobi fields and of the number of
periodic orbits, one can probably obtain logarithmic improvements to the remainder estimates of Theorem \ref{main 1}; see \cite{SXZh17, CG21} among many
papers on logarithmic improvements  on related problems.

\subsection{\label{2TERMINTRO} Two term asymptotics}
We briefly indicate the results of  the long time analysis in \cite{Z22}. 
There, Theorem \ref{main 1} is strengthened to a   two
term asymptotics roughly of the form, $$C_{n,d} a^0_I(H, \psi) \lambda^{\frac{n+d}{2}} + Q_H(\lambda) \lambda^{\frac{n+d}{2} -1} + 
o( \lambda^{\frac{n+d}{2} -1}), $$ where $Q_H(\lambda)$ is a bounded, oscillatory function. To be more precise, we prove  somewhat weaker asymptotic inequalities of the type proved for the pointwise Weyl law  by Yu. Safarov \cite{SV}. 
In statement  \eqref{APERIODIC} of Theorem \ref{main 1}, the second $Q_H$ term vanishes due to aperiodicity of $G^s_H$. 
 In general,  $Q_H(\lambda)$   may be continuous or discontinuous, depending  on the long time
dynamics of the geodesic flows. When it is discontinuous,   the  two-term asymptotics imply that the jump estimates 
above are sharp.  Recently, E. L. Wyman and Y. Xi have proved a two-term asymptotics for the $c=0$ Kuznecov formula (i.e. integrals
of a fixed function $f$ against restricted eigenfunctions) \cite{WX22}.

We give two examples to illustrate the signficance of the continuity of $Q_H(\lambda)$ when it is non-zero.
A model case is supplied by totally geodesic subspheres $\Ss^d \subset \Ss^n$ of spheres (see \cite{Z22}). When $ n=2$
and $H = \gamma$ is the equator, the standard basis $\{Y_N^m\}$ of spherical harmonics has the property that
$Y_N^m |_{\gamma}$ has a single non-zero Fourier coefficient. Its size depends on the ratio $\frac{m}{N}$, illustrating the purpose  of the constraint $|c N -  m | < \epsilon$. The edge case  $m = N$ corresponds to highest weight
spherical harmonics, which are special cases of Gaussian beams (see Section \ref{GB0} for background).
The  Nth Fourier coefficient of the restriction  to a stable elliptic closed geodesic $\gamma$  of a general  Gaussian beam $\{\phi_N^{\gamma}\}_{M=1}^{\infty} $ of frequency $\sim N$  has the size of its $L^2$ norm, $||\phi_N^{\gamma}||_{L^2(\gamma)}$. As this shows,   universal bounds
on individual Fourier coefficients (i.e. without any assumptions on $(M, g, H))$  are the same as universal bounds on the restricted $L^2$ norm (see \cite{BGT} for the relevant results). On the other hand,
restrictions of other spherical harmonics $Y_N^m |_{\gamma}$ have again just one non-zero Fourier coefficient but
in general it is $O(1)$ when $\frac{m}{N} = c < 1$. More drastically,
 if we restrict $\phi_N^{\gamma}$ to 
another geodesic $\gamma' \not= \gamma$,  the restricted Fourier coefficients and $L^2$ norms are exponentially
decaying in $N$. 

A more general example when $\dim M=2$ consists of general closed geodesics.  The remainder estimates are sharp when one can construct a Gaussian beam
along the closed geodesic. Such a Gaussian beam exists along the unique rotationally invariant closed geodesic of a convex
surface of revolution, but does not exist along a hyperbolic closed geodesic for an hourglass of revolution.    In the latter case, $Q_H(\lambda)$ is continuous  and does
not contribute to the jumps. The results of \cite{Z22} prove  such heuristic statements.

\subsubsection{\label{GB0} The role of Gaussian beams} 

Maximal Fourier coefficients for $c=1, d=1$ arise when  there exist Gaussian beams along elliptic  closed geodesics. The classical examples are  highest weight spherical harmonics, which are roughly  of the form, $\phi^{n, \gamma}_{N}(s, y)
= C(n, 1, N)  e^{\frac{2 \pi i N s}{L}} e^{- N |y|^2/2}$  of eigenvalue $-\lambda^2 \sim N^2$; they oscillate along a stable elliptic closed geodesic $s \in \gamma$ of length
$L$,  and have Gaussian decay in the normal directions $y$. Gaussian beams are the  most highly localized eigenfunctions, both in phase space $S^*M$ and  in configuration space $M$; they    concentrate in tubes of radius $\lambda^{-\half} \sim N^{-\half}$ around the phase space geodesic. 
Their Fourier coefficients concentrate at the `edge' $\mu_k = N$, so they  are a $c=1$ phenomenon.  Restrictions of Gaussian beams  do not contribute strongly to the asymptotics
of $J^c_{\epsilon,  H}(\lambda)$ for $c < 1$, since the Fourier coefficients of   $\phi_N^{2, \gamma} $ concentrate at the edge. They are studied in detail in 
\cite{Z22}.

\subsubsection{Quantum Birkhoff normal form analysis when $H$ is a closed geodesic}  For  $d=1$, the  eigenvalues
of $|\frac{d^2}{ds^2}|^{\half}$ form the arithmetic progression $2 \pi \Z$ and have multiplicity at most $2$. When $H$ is a closed geodesic,  it is possible to fix $n$ and 
study the sums $\sum_j \rho(\lambda - \lambda_j) |\int_H \phi_j e_n |^2. $  In higher
dimensions, the eigenvalues are generically distributed uniformly modulo one  \cite{DG75}  and can have various types of multiplicities. 

In  the case of a closed geodesic, one could use quantum 
Birkhoff normal form techniques to study Fourier coefficients; such techniques do not seem available for higher dimensional totally geodesic submanifolds.

 One could also let $H$ be an arc of a non-closed geodesic, but there never exist maximal jumps 
for restrictions to such geodesic arcs and in applications they do not seem important.

\subsubsection{Submanifolds with non-degenerate second fundamental form}
 
In this article, we only study   totally geodesic submanifolds $H \subset M$ in the case $c =1$. These are the extremal cases 
for Kuznecov-Weyl asymptotics, and have been the 
 focus of eigenfunction restriction problems (e.g. \cite{G79, BGT, CG21,  M16, T09}. But the case of
 manifolds with non-degenerate second fundamental form  and with $c=1$ are also important, and in addition are generic. For instance, horocycles and distance
 spheres are non-degenerate and are fundamental in the theory of automorphic forms. The proof of Proposition \ref{N1PARAM}
 breaks down at the steps  in Section \ref{HTG} where $\exp_M^{-1}$ and $\exp_H^{-1}$ are equated on $TH$ and for the same
 reason \eqref{GEO1} is more complicated than \eqref{GEO2}. In the case of a hypersurface  $H$ with non-degenerate second fundamental form  phase function in Proposition \ref{N1PARAM} has a  degeneracy of fold type when $c=1$  rather than the collapse of $N^* H$ along $\rm{Diag}(H \times H)$. The
 fold singularity when $(n,d) = (2,1)$ and $(M,g)$ is a finite area hyperbolic surface with cusps is responsible for the analytic results
 in \cite{Wo04}. It would be interesting to generalize the results of this article (and the case $c< 1$ in \cite{WXZ21}, in which $H$
 is a general submanifold) to the case $c=1$ and $H$ has non-degenerate second fundamental form.  Explicit examples where one can
 expect relative extremals for Kuznecov-Weyl asymptotics
 are non-equatorial  latitude spheres $\Ss^d \subset \Ss^n$. 
 
 \subsection{Acknowledgements} 
 
 This article began as a collaboration with E. L. Wyman and Y. Xi, continuing our work in \cite{WXZ20, WXZ21}, and owes much to the many
 discussion we have had on Fourier coefficient Kuznecov formulae.   
\section{Geodesic geometry} 
 In this section we discuss the geodesic geometry underlying the Kuznecov-Weyl asymptotics, in particular the
 role of \eqref{GEO1}. We begin with a short list of examples of $(M, g, H)$ where $H \subset M$ is totally geodesic.
 A study of Kuznecov-Weyl asymptotics in these examples illuminates the basic question of when the jumps in 
 Corollary \ref{theo JUMP} are of maximal size.

 \subsection{\label{MODSECT} Model examples of totally geodesic $H \subset M$} Closed geodesics
 of Riemannian manifolds are always totally geodesic submanifolds of dimension $d=1$ and many articles
 are devoted to norms and Fourier coefficients of restrictions. See for instance \cite{P32, G83, GRS17,   M16}.
 
  A generic Riemannian manifold does not possess even the germ of a totally geodesic
 submanifold of dimension $d > 1$. Hence we provide a number of  well-known  models which do have such submanifolds.
 
 \begin{itemize}
 \item Totally geodesic subspheres $\Ss^d \subset \Ss^n$ of spheres.  The jump estimates are shown to be sharp for certain cases of $(n,d)$ in 
\cite{Z22}.  Spheres are special in that they are compact rank one symmetric spaces, and the eigenspaces of 
$\Delta$ are spanned in an appropriate sense by Gaussian beams along closed geodesics  (see Section \ref{GB0}). The leading term and
second term  of the Kuznecov-Weyl sums reflect the different types of    restriction of Gaussian beams of $\Ss^n$ to $\Ss^d$, in particular of Gaussian beams of $\Ss^n$ along closed geodesics of $\Ss^d$ versus Gaussian beams along closed geodesics transverse to $\Ss^d$. 
Spheres also illustrate the need for the sum $\sum_{\ell}$ over repeated eigenvalues in an eigenspace
in \eqref{JDEFpsi}. In some cases of $(n,d)$, there exist classical results on asymptotics of Legendre functions which 
also prove the sharpness of the  jump estimates.  \bigskip

\item Totally geodesics submanifolds of other compact rank one symmetric spaces. For instance, the results for
 subspheres $\Ss^d \subset \Ss^n$ of spheres generalize to sub-projective spaces ${\mathbb C} {\mathbb P}^d \subset {\mathbb C}{\mathbb P}^n$. These are examples of Zoll manifolds, all of whose geodesics are closed, and their eigenspaces are spanned by Gaussian beams. One would not expect maximal jumps to occur on general Zoll manifolds, where the Gaussian beams are only quasi-modes and not actual eigenfunctions. \bigskip

\item  Maximal flats of higher rank compact symmetric spaces or of compact  locally symmetric quotients  are
totally geodesic. The asymptotics of  eigenfunctions of   higher rank compact symmetric spaces  are  studied in \cite{G79}, although not the Fourier coefficients of restrictions.  $L^p$ norms of eigenfunctions on higher rank
locally symmetric quotients are studied in \cite{M15}. \bigskip

\item QCI (Quantum Completely integrable) systems \cite{TZ03, T09}. Quantum integrability means that $\Delta$ commutes with $n = \dim M$ 
independent first order pseudo-differential operators. Compact symmetric spaces are QCI and compact locally
symmetric quotients of rank $r$  are partially QCI ($\Delta$ commutes with $n-r$ additional operators).  An open question is whether there exist examples other
than compact rank one symmetric spaces where  jumps achieve their
maximal growth, since joint eigenfunctions concentrate on level sets of the moment map and as in \cite{TZ03, T09}, the eigenfunctions
which concentrate on singular levels have large restricted $L^p$ norms.

  Examples not already discussed include ellipsoids $\ecal_n \subset \R^{n+1}$
of the form $\sum_{j=1}^{d} x_j^2 + \sum_{j= d+1}^{n+1} \frac{x_j^2}{a_j} $ where $\{1, a_j\}_{j=d+1}^{n+1}$ are independent over $\Q$.
The subsphere $\Ss^{d-1} \subset \ecal_n$ is totally geodesic and $G^t_{\Ss^{d-1}}$ is periodic but $G^t_{\ecal_n}$ is aperiodic.  \bigskip

\item Riemannian products $M = H \times K$; $H \times \{k\}$ or $\{h\} \times K$ is  totally geodesic for any $h\in H, k \in K$. For instance if $M = \Ss^n \times \Ss^d$, the geodesic flow is periodic on $H$ but
not on $M$. Flat tori are often products of lower dimensional tori. Maximal jumps are never achieved on product manifolds (see \cite{Z22}).  \bigskip

\item Warped product metrics $M = K \times_w H$. Given  metrics $h$ on $H$ and $k$ on $K$, a warped product has the form $k \oplus  w h$  where 
$w: K \to \R_+$ is a positive smooth function. One often views $M $ as a bundle over $K$ with fiber $H$. A submanifold $S \times K \subset M$ is totally geodesic
in $M$ if and only if $S$ is totally geodesic in $H$ and in particular each submanifold $H \times \{k\}$ is totally geodesic.  More generally, Riemannian submersions with
totally geodesic fibers are examples with totally geodesic submanifolds. \bigskip

 \end{itemize}

\subsection{\label{WFSECT} Wave front calculations}

In this section,  we generalize the results of \cite[Sections 3-5]{WXZ21} on the compositions of canonical relations relevant to smoothed Kuznecov-Weyl sums \eqref{NpsirhoDEF}  to the case $c=1$. Much of the  analysis in the case $c=1$ is almost identical
to that for $c < 1$,  and the proofs are omitted 
when they are essentially the same for $c < 1$ in \cite{WXZ21}  and $c=1$. However, it is necessary to repeat some of the calculations because the order of $S(t, \psi)$ \eqref{SpsiDEF} at $t=0$ in Lemma \ref{ORDERLEM}  is very different from the order in the case
$c <1$, and we need to track down the change in order. 

A second major  change is the existence of a blow-down singularity at $s=0$, leading to the singularity in the leading coefficient \eqref{acHpsi1}, viewed  as a distribution
on the test function $\psi$.  The calculations in this section are used to explain the relevance of the equation \eqref{GEO1} to Kuznecov-Weyl asymptotics and are also used to
prove \eqref{APERIODIC} of Theorem \ref{main 1}.

The analysis in \cite{WXZ21} for the case $c <1$ was based on ladder theory for Fourier integral operators, in the sense of \cite{GU89}. The geodesic
geometry arises in the analysis of the compositions of the Fourier integral operators in Section \ref{WFSECT}. 

The smoothed  Kuznecov-Weyl sums \eqref{NpsirhoDEF} are Fourier dual to traces \eqref{SpsiDEF} arising from the following
operators on  $C^{\infty}(M \times H)$, 
\begin{equation} \label{PQ} \left\{\begin{array}{l} P : = P_M: =   \sqrt{-\Delta}_M \otimes I,\;\;\; \qquad P_H = I \otimes  \sqrt{-\Delta_H}
 , \\ \\ 
 Q_1:=  \sqrt{-\Delta}_M \otimes I -  I \otimes \sqrt{-\Delta}_{H} =  P_M - P_H. \end{array} \right. \end{equation}

As discussed at length in \cite{WXZ21}, 
the system $(P, Q_1)$ is elliptic;  $Q_1 $ is  a non-elliptic first order pseudo-differential operator of real principal type with characteristic variety, \begin{equation}
\label{CHARQ} \Char(Q_1): \{(x, \xi, q, \eta) \in T^*M \times T^*H:  |\xi|_g - |\eta|_{g_H} = 0\}. \end{equation} 
Here, $g_H$ denotes the restriction of $g$ to $TH$.
 
Given $\psi \in \scal(\R)$ with $\hat{\psi} \in C_0^{\infty}(\R)$, we  define the `fuzzy ladder projection',
\begin{equation} \label{FLP}
	\psi(Q_1) \;: L^2(M \times H) \to L^2(M \times H), \;\;
	\psi(Q_1) \;= \int_{\R} \hat{\psi}(s) e^{i s Q_1} ds
\end{equation}

Then, the trace \eqref{SpsiDEF} is given by,

\begin{align} \label{St}
S^1(t, \psi) &= \Pi_*(\Delta_H \times \Delta_H)^* (\gamma_H \otimes I ) e^{i t P} \psi(Q_1) (\gamma_H \otimes I)^* \\
\nonumber
&= \sum_{j,k} e^{it \lambda_j} \psi(\lambda_j - \mu_k) \left| \int_H \phi_{j,k}(x,x) dV_H (x) \right|^2.
\end{align}
They are Fourier dual to the \eqref{NpsirhoDEF} in the sense that, 
\begin{align} \label{FTFORM}S^1(t, \psi) = 
{\mathcal F}_{\lambda \to t} d N^{1} _{\psi, H  }(t) 
\end{align}
The composition theory of \eqref{St} is discussed in Section \ref{WFSECT} below.

The following Lemma is analogous to \cite[Lemma 3.1]{WXZ21} and the proof is the same.
\begin{lemma} \label{WFQc}
$\psi(Q_1)$ of \eqref{FLP} is a Fourier integral operator in the class $ I^{-\half}((M \times H) \times (M \times H)), {\ical_{\psi}^c}')$ with canonical relation
\begin{multline*}
	\ical^c_{\psi} :=  \{(x, \xi, q, \eta; x', \xi',  q', \eta') \in  {\Char}(Q_1)  \times  {\Char}(Q_1) : \\
	\exists s \in \supp(\hat{\psi}) \text{ such that } G^{ s}_M \times G^{-s}_{H}(x, \xi, q, \eta) = (x', \xi', q', \eta') \}.
\end{multline*}
The symbol of $\psi(Q_1)$ is the transport of $(2 \pi)^{-\frac{1}{2}} \hat{\psi}(s) |\dd s|^{\half} \otimes |\dd \mu_L|^{\half}$ via the implied parametrization $(s,\zeta) \mapsto (\zeta, G_M^{s} \times G_H^{-s}(\zeta))$, where $\mu_L$ is Liouville surface measure on $\Char(Q_1)$. 
\end{lemma}

We  introduce a second smooth cutoff $\rho \in \scal(\R), $ with $\hat{\rho} \in C_0^{\infty}$ and define
$$
	\rho(P - \lambda) = \frac{1}{2\pi} \int_{\R} \hat{\rho}(t)  e^{- it \lambda} e^{i t P} dt.
$$
By Fourier inversion, \begin{equation} \label{PQ2} \begin{array}{lll} 
	\rho(P- \lambda) \psi(Q_1) & = &  \frac{1}{2\pi} \int_{\R}  \hat{\rho}(t)  e^{- it \lambda} e^{i t P} \psi(Q_1) \, dt \\ &&\\
	& =& 
	 \frac{1}{2\pi} \int_{\R}  \int_{\R} \hat{\rho}(t) \hat{\psi}(s)  e^{- it \lambda} e^{i t P} e^{- is Q_1}  \, ds dt. \end{array}
\end{equation}

As is well-known \cite{DG75}, $e^{it P}
\in I^{-\frac{1}{4}}( \R \times M \times M, \wt{Graph}(g^t)\}$ where $\wt{Graph}(g^t) = \{(t, \tau, g^t(\zeta), \zeta): \tau + \sigma_P(\zeta) = 0\}$
is the space-time graph of the flow.
For simplicity of notation we denote $\zeta = (\zeta_1, \zeta_2) \in T^*(M \times H)$.  Since the canonical relation of $e^{itP}$ is
the graph of the bicharacteristic flow of the symbol $\sigma_P$ of $P$ on $T^*(M \times H)$, the composition theorem for Fourier integral operators gives,

The following Lemma is analogous to \cite[Lemma 3.4]{WXZ21}. 
\begin{lemma} \label{WFsPIc} $e^{it P} \psi(Q_1): L^2(M \times H) \to L^2(\R \times M \times H)$ is a Fourier integral operator in the class $I^{-\frac{3}{4}}((\R \times M \times H) \times (M \times H), {\ccal^1_{\psi}}')$, with canonical relation
\begin{multline*}
\ccal^1_{\psi} := \{(t, \tau, G_M^{s + t} \times G_H^{-s}(\zeta), \zeta) \in T^* \R \times  \Char(Q_1)  \times  \Char(Q_1): \\ 
s \in \supp(\hat{\psi}), \ \tau + |\zeta_M|_g = 0\}
\end{multline*}
In the natural parametrization of $\ccal_\psi^1$ by $(s,t,\zeta) \in \supp \hat \psi \times \R \times \Char(Q_1)$ given by
$$
(t, - |\zeta_M|_g, G_M^{s + t} \times G_H^{-s}(\zeta), \zeta),
$$
the symbol of $e^{itP} \psi(Q_1)$ is
$(2 \pi)^{-\frac{1}{2}} \hat{\psi}(s)|ds|^{\half} \otimes |\dd t|^{\half} \otimes |\dd \mu_L|^{\half}$,
where $\mu_L$ is Liouville surface measure on $\Char(Q_1).$
\end{lemma}

\begin{remark} Geometrically, this wave front relation consists (initial, terminal) data  of pairs of geodesic arcs, an $H$ arc of length
$s$ and an $M$-arc of length $t +s$, the only constraint on the initial and terminal covectors being that the $H$ initial
(resp. terminal) covectors have the same length as the $M$ initial (resp. terminal) covectors.
\end{remark}

As in \cite{WXZ21} for $c < 1$,  $e^{it P} \circ \psi(Q_1)$ is a transversal composition, and therefore its order is the sum of the
order $\frac{-1}{4}$ of $e^{it P}$ \cite{DG75}  and the order $-\half$ of $\psi(Q_1)$ (Lemma \ref{WFQc}).

To reduce to $H$, we introduce
the restriction operator,
$	\gamma_H \otimes I: C(M \times H) \to C(H \times H),$
and define 
\begin{equation} \label{LRLEM}
	\gamma_{\R \times H \times H} \circ (B_{\epsilon}(x,D) \otimes I) e^{it P} \psi(Q_1) \circ \gamma_{H \times H}^*
\end{equation}

Here, $B_{\epsilon}(x, D)$ is a cutoff operator away from normal directions. 
We refer to \cite{WXZ21} for a discussion of such cutoff operators,  but note that in the totally geodesic case
the cutoff away from tangential directions is un-necessary.  
For fixed $\epsilon >0,$ we define the cutoff operator $\chi^{(n)}_{\epsilon}(x,D) = Op(\chi_{\epsilon}^{(n)}) \in Op(S^0_{cl}(T^*M))$ has its homogeneous symbol  $\chi^{(n)}_{\epsilon}(x,\xi)$ supported in an $\epsilon$-conic neighbourhood of $N^*H$ with $\chi^{(n)}_{\epsilon} \equiv 1$ in an $\frac{\epsilon}{2}$ subcone. We put, 
$B_{\epsilon}(x, D) = I - \chi^{(n)}_{\epsilon} (x, D).$

The following is the analogue of \cite[Lemma Proposition 4.7]{WXZ21} when $c=1$.
\begin{proposition} \label{LRLEMPROP2} The wave front relation of \eqref{LRLEM} is given by,
\begin{multline*}
\Gamma^1_{\psi, \epsilon} := (\pi_{\R \times H \times H} \times \pi_{H \times H}) \ccal_{\psi, \epsilon}^1 \cap (T^*\R \times T^*_{H } M \times T^*H \times T^*_H M \times T^*H ) \\
= \{(t, \tau, \pi_{H \times H} \zeta, \pi_{H \times H} (G^{ s +t}_M \times G^{-s }_{H}) (\zeta)) : |\zeta_M|_g + \tau = 0,\\
\zeta \in \Char Q_1 \cap T^*_{H \times H} (M \times H), \ G^{ s +t}_M(\zeta_M) \in T^*_H M \\
 (1- \chi_{\epsilon}) (G^{ s +t}_M(\zeta_M)) \not=0,  \ s \in \supp \hat \psi \} \\
\subset T^*\R \times (T^*H \times T^*H \times T^*H \times T^*H).
\end{multline*}
Moreover, on the support of the cutoff operator $B_{\epsilon}(x, D)$ away from $N^*H$, $\Gamma^1_{\psi,\epsilon}$
is a Lagrangian submanifold and 
the `reduced' Fourier integral operator 
\begin{equation} \label{LRLEM2}
	\gamma_{\R \times H \times H} \circ (B_{\epsilon}(x,D) \otimes I) e^{it P} \psi(Q_1) \circ \gamma_{H \times H}^*
\end{equation}
belongs to  the class 

$$
	I^{\rho(m, d)} (\R \times (H \times H) \times (H\times H), \Gamma^1_{\psi,\epsilon}),
$$
with 
$$\rho(m, d) = \ord  e^{it P} \psi(Q_1)  + \half (n-d) + 2 d +\half -\half  (4 d+ 1) = \ord  e^{it P} \psi(Q_1)  + \half (n-d) . $$

\end{proposition}
The statement and proof are essentially the same as in the case $c < 1$ of  \cite[Lemma Proposition 4.7]{WXZ21} and
we therefore omit the proof and refer there for the details.

Next we pullback under the diagonal embedding  to obtain the following is analogue of \cite[Proposition 5.2]{WXZ21}. 
  We use the notation,  $$\left\{ \begin{array}{l} \zeta_H = (q, \eta), \zeta_H' = (q', \eta'),\; G^s(q, \eta) = (q', \eta') \\ \\ \zeta = (\zeta_M, \zeta_H) =  (x, \xi, y, \eta'')  \in {\rm Char}(Q_1), (x,  \xi) \in T^*_H M,   G^{ s +t}_M(x, \xi) \in T^*_H M\end{array} \right. $$  

\begin{proposition} \label{M0PROP} The wave front relation of  $(\Delta_H \times \Delta_H)^*\Gamma^1_{\psi, \epsilon} \subset T^*\R \times T^*H \times T^*H$ is given by,

\begin{multline*}
(\Delta_H \times \Delta_H)^*\Gamma^1_{\psi, \epsilon} = \{(t, \tau, (q,\eta -\pi_H \xi) ; (q', \eta' - \pi_H  \xi')) \in T^*\R \times T^*H \times T^*H: \\
\xi \in T_q M, \xi' \in T_{q'} M, \exists (s, \sigma): G^s(q, \eta) = (q', \eta'), G^{t +s}(q, \xi) = (q', \xi'),\\ 
\tau = - |\eta| = - |\xi|, |\eta'| = |\xi'|\}.
\end{multline*}
.\end{proposition}

\begin{remark} This wave front set consists of analogues for bi-angles of geodesic loops. A geodesic loop 
of length $T$ at a point $x \in M$ is given by a geodesic arc $\exp_x t \xi$ satisfying $\exp_x T \xi = x$. Unlike
a closed geodesic, the initial and term directions do not have to be the same. A bi-angle is the analogue of a closed
geodesic but the `bi-angle-loop'  consists of two geodesic arcs, an $M$-arc and an $H$-arc  from $q$ to $q'$, with 
no constraint that the projection of the initial or terminal directions of the $M$ arc agree with those of the $H$ arc.

\end{remark}

\begin{proof} The calculation is similar to that of  \cite[(1.20)]{DG75} for the pullback to the `single diagonal' in $M \times M$. The pullback to the `double-diagonal' 
 $\Delta_{H \times H} \subset H \times H \times H \times H$ subtracts  the two covectors at the same base points in the double-diagonal.

  In terms of the above notation,
 \begin{multline*}
(\Delta_H \times \Delta_H)^*\Gamma^1_{\psi, \epsilon} = \{(t, \tau, (q,\eta -\pi_H \xi) ; (q', \eta' - \pi_H  \xi')) \in T^*\R \times T^*H \times T^*H: \exists s \\ 
(t, \tau, (q, \eta), (q, \pi_H \xi), (G^s(q, \eta), \pi_H G^{t + s}(q, \xi)))  \in \Gamma^1_{\psi, \epsilon}\} \\ \\
= 
\{(t, \tau, (q,\eta -\pi_H \xi) 
\exists (s, \sigma, \pi_{H \times H} (x, \xi, y, \eta'') , \pi_{H \times H} (G^{ s +t}_M \times G^{-s }_{H}) ((x, \xi, y, \eta'') )) \in \Gamma^1_{\psi, \epsilon}, \\
(q, q, q', q') =  (x, y, \pi G_M^{s + t}(x, \xi), \pi G_H^{-s}(y, \eta'') ), \\
(\zeta_H, \zeta_H') =   (\Delta_H \times \Delta_H)^*  (\pi_{H \times H} \zeta, \pi_{H \times H} (G^{ s +t}_M \times G^{-s }_{H}) (\zeta))) \} \end{multline*}

Indeed, on   the base $M \times H$, the pullback restricts to the double diagonal,
\begin{multline*}
	\Delta_H \times \Delta_H (q, q') =(q,q, q',q') = \pi  \left(\pi_{H \times H} \zeta, \pi_{H \times H} (G^{ s +t}_M \times G^{-s }_{H}) (\zeta)\right) \\
\iff q= x = y, \;\; q' =  \pi G^{s + t}_M(x, \xi) = \pi G^{-s}_H (y, \eta'').   \end{multline*}
  Moreover, on  the (co-)vector level, $\eta'' = \eta$  and,
   $$\begin{array}{lll} (\zeta_H, \zeta_H') & = &  ( \Delta_H \times \Delta_H)^*\pi_{H \times H} (\zeta, (G_M^{t + s} \times G_H^{-s})^* (\zeta)), 
   \\&&\\ &\iff & 
   \pi_H (q, \xi) = (q, \eta) ,  \; \pi_H G_M^{t + s} (q, \xi) = (q', \eta') = G_H^{-s} (q, \eta).   
   \end{array} $$  
   These conditions imply that $(q, \xi) \in S^* H$, since $\pi_H \xi = \eta$ and $ |\xi| = |\eta|$.
\end{proof}

The next step is to pushforward under $\Pi: \R \times H \times H \to \R$, which results in `closing' the bi-angle-loop wave front set to the set of 
`closed bi-angles'.

\begin{proposition} \label{MAINFIOPROP} The pushforward wave front set is given by, \begin{align} \nonumber
\Lambda^1_{\psi} &:= \Pi_*(\Delta_H \times \Delta_H)^*\Gamma^1_{\psi, \epsilon} \subset T^*\R \\
\label{PUSHPULL} &= \{(t, \tau) \in T^*\R: \exists \{(q, \eta): G^t(q, \eta) = (q, \eta)\} \end{align}
\end{proposition}

\begin{proof} As in  \cite[(1.21)]{DG75}, the pushforward
operation erases  points of $$ (\Delta_H \times \Delta_H)^*\Gamma^1_{\psi, \epsilon} = \{(t, \tau, (q,\eta -\pi_H \xi) ; (q', \eta' - \pi_H  \xi'))\} $$
unless $\eta -\pi_H \xi) =  \eta' - \pi_H \xi' = 0$ and the output for such vectors is $(t, \tau)$. Since $H$ is totally geodesic,
and $\pi_H \xi = \eta, $ one has $G^{t +s} (q, \eta) = G^s (q, \eta)$.  Cancelling the $s$ factors gives the result.

\end{proof}

\subsection{Cleanliness issues}
We recall that except for the last statement of  Theorem \ref{main 1} and Theorem \ref{main 2}, it is assumed 
that $\rm{supp} \hat{\psi} , \rm{supp} \hat{\rho} \subset (-r_0, r_0)$. The wave front calculations above explain the
difficult cleanliness issues that arise if we drop the assumption that $ \rm{supp} \hat{\rho} \subset (-r_0, r_0)$. 
In the intermediate wave front set calculations, one would need that 
the set of arcs from $q$ to $q'$ of arbitrary length form clean submanifolds. This is obviously true for $H$ arcs or $M$-arcs of lengths in $(-r_0, r_0)$, where the arc from $q$ to $q'$ is unique but is rarely true for long arcs. 
In particular, there are problems if  $q,q'$ are  conjugate points along an $H$ arc.

It is possible that the wave front relations are not Lagrangian submanifolds until the last step, where the the set of
arcs consists of closed geodesics. In that case, we would need that the fixed point sets $\rm{Fix}G^t_H$ are all 
clean in the sense of \cite{DG75}; this is substantially simpler than cleanliness of the arc-sets in the intermediate steps, but is still non-generic.

\subsection{The order of \eqref{SpsiDEF}  at $t=0$}
We now  calculate the order of \eqref{SpsiDEF} after  further composing with the diagonal pullback to
 $(H \times H) \times (H \times H)$ and then the pushforward to $\R$.     The trace is the composition defined by the fiber product diagram
of $\Gamma^1_{\psi, \epsilon}$ with the conormal bundle of the diagonal of $H \times H$ (see \cite[(8.6)]{WXZ21}). 
Note that the order at $t=0$ is quite different from the case $c < 1$ in \cite{WXZ21},
 and justifies including the  material in the previous section.

In a neighborhood of  $t=0$, \emph{?}qref{SpsiDEF} is a Lagrangian distribution and we may calculate its order using
the calculus of Fourier integral operators despite the blow-down singularity at $s=0$.

\begin{lemma} \label{ORDERLEM}   As in \cite[(5.10)]{WXZ21}, if $\hat{\psi} =0$ in an interval around $s=0$, then  the order of $S(t, \psi)$ is 
 $\frac{-3}{4} + \half (n-d)  + \half \dim \gcal_1^T$, where $\gcal_1^T = \rm{Fix}(G^t_H)$. From the fact that $\dim \gcal_1^0 = \dim S^*H = 2d -1$,  $$\rm{ord} S(t, \psi) |_{t =0} = -\frac{3}{4} + \half(n-d)
+ d - \half = -\frac{5}{4}  + \frac{n+d}{2}. $$
Recall that $I^{\frac{\nu}{2} - \frac{1}{4}}$ has symbols of order $s^{\frac{\nu - 1}{2}}$.  \end{lemma}
 
 \begin{proof}
 The composition is clean away from the singularity at $s=0$. The  order and composition are therefore  computed precisely as in 
\cite{DG75} and \cite[(5.7)-(5.8)]{WXZ21}.  However the excess $e(0)$ is different and results in the different order.
\end{proof}

\section{\label{main1SECT} Proof of Theorem \ref{main 1} and Theorem \ref{main 2} using H\"ormander parametrices}

In this section, we use the H\"ormander  parametrices of the wave proof   to proveTheorem \ref{main 2} for $\psi$ such that $\hat{\psi} = 0$ in some
interval around $0$. Although this does not clarify the regularization at $s=0$, its phase is simple and exhibits  the blow-down singularity responsible
for the singular coefficient in Theorem \ref{main 1}. In fact, it is precisely the phase of a certain Bessel-type integral
that we refer to as a ``double-Bessel function" \eqref{DB}. In Section \ref{MODELSECT} we build on this parametrix approach to construct model
oscillatory integrals. We then reduce the actual problem to the model case. The H\"ormander  parametrix method is also used in \cite{WXZ21} and 
makes it easy to compare with the case $c < 1$ (see Section \ref{COMPSECT}).

\subsection{\label{HORPAR} H\"ormander small time parametrices}

In this section we review the H\"ormander parametrix of the half-wave kernel and use it to derive the following,

\begin{proposition}\label{N1PARAM}
Let $(M, g)$ be any compact Riemannian manifold of dimension $n$, and let $H \subset M$  be a totally geodesic
submanifold  of dimension $d \leq n-1$. Then, there exists a semi-classical amplitude $\wt A( \langle y,\omega \rangle, \langle y, \wt \omega \rangle,  q, y,  \omega,   \wt \omega) $ of order zero such that,
\begin{equation} \label{FV} \begin{array}{l} N^{1} _{\psi, \rho, H  }(\lambda) = \lambda^{n + d -2} \int_0^{\infty}
\int_H \int_{T_q H} \int_{S_q^* H} \int_{S_q^* M} \hat{\psi}( \langle y, \wt \omega \rangle) \hat{\rho}( \langle y, \omega  \rangle)   e^{i \lambda  \langle y, \omega + \wt \omega \rangle } \\ \\ \wt A( \langle y,\omega \rangle, \langle y, \wt \omega \rangle,  q,  \omega,  c \wt \omega) \;  dV_H(y)  dV_H(q)  d S_q(\omega)   d S_q(\wt \omega), \end{array}\end{equation}
When $d=1$, $S_q^*H = \{\pm e_1\}$ and the integral is a sum over $\pm$.  
\end{proposition}

\begin{proof}
On any Riemannian manifold, we  may construct  small time parametrix for $U(t) = e^{i t \sqrt{-\Delta}}$ an oscillatory
integral of the form,

\begin{equation} \label{Ut} U(t, x, y) = \int_{T^*_x M} e^{i \langle \exp_x^{-1} (y), \xi \rangle} e^{it |\xi|} A(t, y, \xi) d \xi,
\end{equation}
where $A(t, x, y, \xi) $ is a homogeneous amplitude of order $0$ and supported in the set $r(x, y) \leq t + \delta$ for any small $\delta> 0$. 
Here, $\exp_x : B_x(\epsilon) \subset T_x M \to M$ is the exponential map and $B_x(\epsilon)$ is a sufficiently small ball so that $\exp_x$
is a diffeomorphism to its image. 
The amplitude
is independent of $x$  and f or  $|t| < \rm{inj}(M,g)$ satisfies, for all  $(y, \xi) \in T^*M$,
\begin{equation} \label{AMP} \left\{ \begin{array}{l} 
 A(0, y, \xi) = 1,  \\ \\ 
A(t, y, \xi ) -1  \in S^{-1}. 
\end{array} \right. \end{equation}

Then, if $\hat{\rho}$ and $\hat{\psi}$ have sufficiently small support so that the parametrix \eqref{Ut} is valid for both $U_M(t), U_H(s)$,
\begin{equation} \label{N1HFORM} \begin{array}{lll} N^1_{H, \psi, \rho}(\lambda)  & = &  \int_{\R} \int_{H \times H} \hat{\rho}(t) e^{it \lambda} \Pi_{H \times H} \Delta_{H \times H}^* \gamma_H \psi(Q_1) e^{it P} \gamma_H^*
dt dV_H(q) dV_H(q')\\ &&\\
& = &  \int_{\R} \int_{\R} \int_{H \times H}  \hat{\psi}(s) \rho(t)  e^{it \lambda} 
U_H(-s, q, q') U_M( t + s, q, q') dV_H(q)    dV_H(q') ds dt  \end{array} \end{equation}
 has the parametrix, 
\begin{equation}\begin{array}{l} 
 \int_{H \times H}  \int_{T_q H } \int_{\R}\int_{\R} \int_{T_q^* M} \int_{T_{q'}^* H} \hat{\psi}(s) \hat{\rho}(t)  
 e^{it \lambda} e^{i \Psi_1}
\wt A_1(t, s, q, q', \xi, \eta) 
 dV_H(q) dV_H(q') ds dt d \xi d \eta.
\end{array} 
\end{equation}
where 
$$ \Psi_1 = \langle  (\exp^M_{q})^{-1} (q'), \xi \rangle + (\exp^H_{q})^{-1} (q'), \eta \rangle + ( t +  s)  |\xi| -   s |\eta|,  $$
and where $$\wt A_1(t, s, q, q', \xi, \eta) = A_M( t+  s , q, q', \xi) A_H( - c s , q, q', \eta).$$

We change variables $\xi \to \lambda \xi, \eta \to \lambda \eta$ and obtain a semi-classical oscillatory integral,
\begin{equation}\label{COV1} \lambda^{n + d}
\int_{H \times H} \int_{\R} \int_{\R} \int_{T_q^* M} \int_{T_{q'}^* H}\hat{\psi}(s) \hat{\rho}(t)  e^{it \lambda} e^{i \lambda \Psi_1 } \wt A(t, s, q, q', \xi, \eta) \;   dV_H(q) dV_H(q') ds dt d \xi d \eta. \end{equation}

To simplify the integral, we  set $\xi = \rho \omega, \eta = \sigma \wt \omega$ and change variables to get the phase,
$$\Psi_2(t, s, q, q', \rho, \omega, \sigma, \wt \omega): =  t(1 - \rho) + s( \rho - \sigma) + \rho 
 \langle (\exp^M_{q})^{-1} (q'), \omega  \rangle + \sigma \langle (\exp^H_{q})^{-1} (q'), \wt \omega \rangle. $$
 We eliminate the pair of variables $dt d \rho$ by stationary phase, which introduces a new factor $\lambda^{-1}$ and
 sets $\rho =1, t =  \langle (\exp^M_{q})^{-1} (q'), \omega  \rangle. $ The Hessian determinant equals $1$, reducing the integral to, 
 \begin{equation}\begin{array}{l} \lambda^{n + d -1}\int_0^{\infty} 
\int_{H \times H}  \int_{\R} \int_{S_q^* M} \int_{S_{q'}^* H}\hat{\psi}(s) \hat{\rho}(\langle (\exp^M_{q})^{-1} (q'), \omega  \rangle)   e^{i \lambda \Psi_2 } \\ \\ \wt A_2( \langle (\exp^M_{q})^{-1} (q'), \omega  \rangle, s q, q', \xi, \eta) \;   dV_H(q) dV_H(q') ds  d S(\omega) \sigma^{d-1} d \sigma d S(\wt \omega),\end{array}
\end{equation}
where $$\Psi_2 = s (1 - \sigma) +  \langle (\exp^M_{q})^{-1} (q'), \omega  \rangle + \sigma \langle (\exp^H_{q})^{-1} (q'), \wt \omega \rangle,$$
and where the new amplitude $\wt A_2$ still satisfies  \eqref{AMP}.

Next we eliminate $d \sigma d s$ by stationary phase  to obtain $\sigma = 1$ and $s = \langle (\exp^H_{q})^{-1} (q'), \wt \omega \rangle$. The Hessian
again equals $1$ and the oscillatory integral reduces to 
and reduce to,

\begin{equation}\label{COV2}  \begin{array}{l} \lambda^{n + d -2} 
\int_{H \times H} \int_{S_q^* M} \int_{S_{q'}^* H}\hat{\psi}( \langle (\exp^H_{q})^{-1} (q'), \wt \omega \rangle) \hat{\rho}(\langle (\exp^M_{q})^{-1} (q'), \omega  \rangle)   e^{i \lambda \Psi_3 } \\ \\ \wt A_3( \langle (\exp^M_{q})^{-1} (q'), \omega  \rangle, \langle (\exp^H_{q})^{-1} (q'), \wt \omega \rangle q, q',  \omega,   \wt \omega) \;   dV_H(q) dV_H(q')   d S(\omega)   d S(\wt \omega), \end{array}\end{equation}

where
$$\Psi_3 (q, q', \omega, \wt \omega): =  
 \langle (\exp^M_{q})^{-1} (q'), \omega  \rangle +  \langle (\exp^H_{q})^{-1} (q'), \wt \omega \rangle, $$
 and where the new phase satisfies
 $$ \wt A_3( \langle (\exp^M_{q})^{-1} (q'), \omega  \rangle, \langle (\exp^H_{q})^{-1} (q'), \wt \omega \rangle q, q',  \omega,   \wt \omega) = 1, \;\;
 (\rm{when}\;\;q = q').$$

\subsubsection{ \label{HTG} $H$ totally geodesic} 
We now use that $H$ is totally geodesic for the first time. Since $q, q' \in H$,  $(exp^M_{q})^{-1} (q')= (exp^H_{q})^{-1} (q'),$ and we may combine terms 
to simplify the phase to,
$$\Psi (q, q', \omega, \wt \omega): =  
 \langle (\exp^H_{q})^{-1} (q'), \omega +   \wt \omega \rangle. $$
 For simplicity of notation, we henceforth denote $$ y  = (\exp^H_{q})^{-1} (q') \in T_q H, r = r_H(q,q'), y = r v$$
 and obtain the universal phase, 
 \begin{equation} \label{PSIDEF} \Psi (y, \omega; \wt \omega): = \langle y,  \omega +   \wt \omega \rangle, \;\; (y \in T_q^*, \omega \in S^*_q M, \wt \omega
 \in S^*_q H). \end{equation}
 Henceforth, we regard $(q, \wt \omega)$ as parameters and consider the oscillatory integral,
 \begin{equation} \label{N1q} \begin{array}{lll}  N^1_{H, \psi, \rho}(\lambda, q, \wt \omega) & = &  \lambda^{n + d -2} 
 \int_{S_q^* M} \int_{T_q H} \hat{\psi}( \langle (y,  \wt \omega \rangle) \hat{\rho}(\langle (y, \omega  \rangle)   e^{i \lambda \langle y, \omega + \wt \omega \rangle } \\ && \\  && \wt A( \langle y, \omega  \rangle, \langle y, \wt \omega \rangle, y,  \omega,   \wt \omega) \;   d y  d S(\omega),   \end{array}\end{equation}
 where $\wt A = 1$ when $y = 0$.
 
 This completes the proof of Proposition \ref{N1PARAM}.
 
 \end{proof}
 
\subsection{Analysis of the phase}

For purposes of comparison to model phases in Section \ref{MODELSECT}, we extract from  \eqref{FV}
the sub-integrals,
 \begin{equation} \label{FVJDEFa} \begin{array}{lll} J^1_{H, M} (\lambda, q,  r v,  \chi, \psi, \rho) &:=& 
 \int_{S_q^* H} \int_{S_q^* M} \hat{\psi}(r \langle v, \wt \omega \rangle) \hat{\rho}(r \langle v, \omega  \rangle)   e^{i r\lambda  \langle v, \omega + \wt \omega \rangle } \\ &&\\ &&\wt A( \langle v,\omega \rangle, \langle v, \wt \omega \rangle,  q, r v,  \omega,  c \wt \omega) \;     d S_q(\omega)   d S_q(\wt \omega). \end{array}\end{equation}
 As will be seen in Section \ref{MODELSECT}, they are closely related to Bessel (or, rather, double-Bessel)
 integrals \eqref{DB} and exhibit the same blow-down singularities. Their relation to \eqref{FV} is given by,
\begin{equation} \label{FVb} \begin{array}{l} N^{1} _{\psi, \rho, H  }(\lambda) = \lambda^{n + d -2} \int_0^{\infty}
\int_H \int_{S_q H} J^1_{H, M} (\lambda, q,  r v,  \chi, \psi, \rho)\;  r^{d-1} dr dV_H(q) dS_q(v). \end{array}\end{equation}


As before, we regard the variables $(q, \wt \omega) \in S^*_q H$ as parameters and the the variables
$(q', \omega) \in H \times S^{n-1}$ as the ``phase variables'' of the oscillatory integral. We change variables as in that section, 
to write \begin{equation} \label{COORDS} q' = \exp_q y= \exp_q (r v), \;\;\; y = rv \in T_q H, v \in S_q H, r > 0, \;\; (y, \omega) \in T^*_q H \times S^*_q M. 
\end{equation} This is
possible since the short-time parametrix is used only for $(q, q') $ near the diagonal $\Delta_H \times H \times H$. It should be kept in mind that $y =0$ or $r=0$  corresponds to the diagonal in $H \times H$.
 As usual, we 
identify $T^*_q H = T_q H$ using the metric $g$ without further comment.

We therefore study the inner integral of Proposition \ref{N1PARAM},  
\begin{equation} \label{udef} \begin{array}{l}u(\lambda, q, \wt \omega): =\int_{T_q H}  \int_{S_q^* M}  \hat{\psi}( \langle y, \wt \omega \rangle) \hat{\rho}( \langle y, \omega  \rangle)   e^{i \lambda  \langle y, \omega + \wt \omega \rangle } \wt A \; dy    d S_q(\omega), \end{array} \end{equation}
where $\wt A  = ( \langle y,\omega \rangle, \langle y , \wt \omega \rangle,  q, y,  \omega,   \wt \omega)$. Recall that this parametrix expression is
only valid for $|t| $ smaller than the injectivity radius of $H$.

The   phase function of \eqref{FVJDEFa} or \eqref{udef} in the coordinates \eqref{COORDS}  is,  \begin{equation}\label{PSI} 
\Psi(\wt \omega; y, \omega): = \langle y, \omega - \wt \omega \rangle : S^*_q M \times T_q H   \times S^*_q H  \to \R. \end{equation}
Here, $q \in H$ is fixed and is not  indicated further in the notation for \eqref{udef};  henceforth, we view the phase \eqref{PSI}
as defined on $S^{d-1}_{\wt \omega} \times \R^d_y  \times S^{n-1}_{\omega} $.


In this section, we analyze the Lagrangian submanifold generated by the phase \eqref{PSI} and, in particular, 
the singularity of its projection at $y = 0$ (i.e. on the diagonal). In Section \ref{MODELSECT}, we show that the 
cubic Taylor approximation to \eqref{PSI} is the normal form phase function generating a Lagrangian submanifold
with a blow-down singularity. 
We follow \cite{DG75} and \cite{HoIV} for background on clean phase functions of homogeneous oscillatory integrals, 
\cite{D74} for background on semi-classical (non-homogeneous) oscillatory integrals with a large parameter, and \cite{G89} for background
on blow-down maps.

We recall (cf. \cite[Page 71]{DG75}) that a phase function $\phi(x, \theta)$ on $X \times \R^N$ is {\it clean} if $$C_{\phi}: = \{(x, \theta) \in X \times \R^N : \nabla_{\theta} \phi(x, \theta = 0\} $$  is a submanifold of $X \times \R^N$ and at each point of $C_{\phi}$
the tangent space is the space of vectors annihilated by $d (\frac{\partial \phi}{\partial \theta_1}), \dots, d (\frac{\partial \phi}{\partial \theta_N})$. The excess
$e$ of the clean phase is  defined so that  $N-e $ is the dimension of the space spanned by these differentials. When the phase is clean,
the map
$$\iota_{\phi} : C_{\phi} \to \Lambda_{\phi}  \subset T^*X, \;\; \iota_{\phi}(x, \theta) = (x, \nabla_x \phi) $$
is a fiber mapping of fiber dimension $e$ over its image $\Lambda_{\phi}$. The Leray measure $d_{C_{\phi}}$ on $C_{\phi}$  is the pullback $\delta_0(\nabla_{\theta} \phi)$ of $\delta_0$ under the map $\nabla_{\theta}  \phi :X \times \R^N \to \R^N$. It is well-defined if $\nabla_{\theta} \phi $ is a submersion.

For  the  phase  \eqref{PSI} with $q$ fixed, 
the role of  $X$ is played by $\wt \omega \in  S^{d-1}$  and the role of the phase variable $\theta \in \R^N$ is played by $\theta = (y_1 \dots, y_{d},  \omega \in S^{n-1}),$ with $N = d + n-1$.  Then,
$$\nabla_{y, \omega} \Psi = (\omega - \wt \omega, y - \langle y, \omega \rangle \omega) \in S^{n-1} \times H \simeq S^{n-1} \times \R^d $$
and the critical set of the phase \eqref{PSI} is given by,
\begin{equation} \label{CPSI} C_{\Psi} = \{(\wt \omega; y, \omega): \omega = \wt \omega, \;\; y = \langle y, \wt \omega \rangle \wt \omega \} \subset S^*_q M \times T_q H
\times S^*_q M. \end{equation}

As the next Lemma shows, the phase \eqref{PSI} fails to be clean (much less, non-degenerate).

\begin{lemma} \label{UNCLEAN2} \begin{itemize} We have, \bigskip

\item (i)\; The phase \eqref{PSI} is a clean phase of excess $e=1$ on the complement of the diagonal $q = q'$, i.e.  for $y \not= 0$.  That is, the kernel of $d \Psi'_{y, \omega} $ on $ T_{\wt \omega, y, \omega} (S^{d-1} \times \R^d \times S^{n-1})$ equals $ T_{\wt \omega, y, \omega}  
C_{\Psi} $ when $y \not= 0$. If $y = rv$, all components of $d_{\wt \omega, v, \omega} \Psi$ are independent on $C_{\Psi}$.   \bigskip

\item (ii) \; The phase \eqref{PSI} is not  clean along the set where $\{y=0\}$ (the diagonal).  The kernel of $d \Psi'_{y, \omega} $ on $ T_{\wt \omega, y, \omega} (S^{d-1} \times \R^d \times S^{n-1})$ jumps at $y =0$ to include
$SN^* \Delta_{H \times H}$,  i.e. 
$\omega \in S^*_q M: \pi_H \omega = 0. $ \bigskip

\item (iii) \;  The map $\iota_{\Psi}: C_{\Psi}  \backslash \{y =0\} \to \Lambda_{\Psi} $ is an $\R^*$ bundle over the zero section $\Lambda_{\Psi} = 0_{T^* S^{d-1}} \subset T^* S^{d-1} $. It has a blow-down singularity over $\{y = 0\}$ (see Section \ref{BDSECT} \bigskip


\end{itemize}
 \end{lemma}

\begin{proof} 

\noindent{\bf Proof of (i)}
\bigskip 

 For   fixed $\wt \omega$, the equation  $y = \langle y, \wt \omega \rangle
\wt \omega$ determines $v = \frac{y}{r} = \wt \omega$ and therefore the slice $C_{\Psi}(\wt \omega)$ with fixed
$\wt \omega$ can be identified with $r \in (0, \infty)$ when $r > 0$. $C_{\Psi}(\wt \omega)$ also contains the point $y = 0$ 
for every $\wt \omega$. It follows that $C_{\Psi} \simeq S^*_q H \times \R_+$ is a manifold with boundary
$S^*_q \times \{0\}$.  Henceforth, we put
$$\Psi^{\wt \omega}(y, \omega) = \Psi(\wt \omega ,y, \omega). $$

To calculate the  covectors
 $d (\frac{\partial \Psi}{\partial \theta_j}) = d \nabla_{y, \omega} \Psi$, we use
 that the phase \eqref{PSI} is symmetric under the diagonal action of $SO(d)$ on $S^{d-1} \times \R^d \times S^{n-1}$, where it acts on the third
factor by $g \cdot \omega = (g \pi_H \omega, \omega^{\perp})$. We fix $q \in H$ and define linear coordinates  $(y_1, \dots, y_d)$ on $T_q H \simeq \R^d$.  Without loss of generality we may assume that $\wt \omega = e_d$. 
 We also 
fix linear 
coordinates $(x_1, 
\dots, x_n)$  on $T_q^* M \simeq \R^n$ and endow $S^*_q M \simeq  S^{n-1}$ with the coordinates $x' := (x_1, \dots, x_{d-1}, x_{d+1},
\dots, x_n) $, with  $x_d = \pm \sqrt{1 - |x'|^2}.$
Then $\omega = (x_1, \dots, x_{d-1},  \pm \sqrt{1 - |x'|^2}, x_{d+1}, \dots, x_n)$. In these coordinates, $e_d = (\vec 0, 1, \vec 0) $ corresponds to $x' = 0$ in $B_1(\R^{n-1})$. Then, $$\omega - \wt \omega = 
 (x_1, \dots, x_{d-1},  \pm \sqrt{1 - |x'|^2} - 1, x_{d+1}, \dots, x_n),\;\;\; (\wt \omega = e_d)$$  so that for $|x'| < 1$ and 
 on the $\pm$ hemisphere $S^{d-1}_{\pm}$ where $ x_d = \pm \sqrt{1 - |x'|^2}$, 
 \eqref{PSI}  reduces to \begin{equation} \label{Psied}  \Psi_{\pm}^{e_d}: =  \sum_{j=1}^{d-1} y_j x_j \pm y_d ( \sqrt{1 - |x'|^2} -1)
=  \sum_{j=1}^{d-1} y_j x_j + y_d (x_d -1). \end{equation}
We note that the choice of $\wt \omega = e_d$ creates the sign asymmetry between $\pm \sqrt{1- |x'|^2} -1$. 

Next, we  calculate the  covectors  $d (\frac{\partial \Psi}{\partial \theta_j}) = d \nabla_{y, \omega} \Psi$ in these coordinates.  Let $\partial_{x_j} =\frac{\partial}{\partial x_j}$.  Since
$\partial_{x_j} \sqrt{1 - |x'|^2} = \frac{x_j}{\sqrt{1 - |x'|^2}}, \;\; j \not= d, $ the critical point equations of the phase \eqref{Psied} on the complement of $x_d = 0$ are, 
\begin{equation} \label{CPE} \left\{ \begin{array}{ll} (i) &  \partial_{y_j}  \Psi^{e_d}= x_j =  0 \; (j = 1, \dots, d-1);\\ &  \\ (ii) &  \partial_{y_{d}}  \Psi^{e_d}=x_d - 1  = 0 \iff x_d =1 \iff   \pm \sqrt{1 - |x'|^2} = 1 \iff +, x' =0,\\ &\\
(iii) & 
 \partial_{x_j} \Psi^{e_d} = y_j + y_d \frac{x_j}{\sqrt{1 - |x'|^2}} =  0 \iff y_j = 0,  j = 1, \dots, d-1;  \\ &  \\ (iv) & 
 \partial_{x_j} \Psi^{e_d} =  y_d \frac{x_j}{\sqrt{1 - |x'|^2}} = 0 \iff y_d =0\; \rm{or}\; x_j =0, \; j = d +1, \dots,n ; \end{array} \right.
 \end{equation}
There is a sign  asymmetry in $\pm$ due to the choice of $\wt \omega = e_d$   in (ii), since there are no solutions for the $-$ sign. 
 In (iii) we use (ii) to simplify the equation.
By (i),   $x_j = 0, $ for $j = 0, \dots, d-1$, hence on the critical set, $x' = (0, 0, \dots, 0, x_d, x_{d+1}, \dots, x_n) $. The critical point equation in 
$y_d$ implies $x_d = \pm 1$ (hence $=1$ as in (ii)  since we have chosen $e_d$ as the base point). 
It follows that $x_{d+1}, \dots, x_{n} = 0$ at a critical point. 
Combining the first critical point equation  with the second critical point equation, one finds that $y_j = 0$ for $j =1, \dots, d-1$. When $y_d = 0$, equation (iv) does not force $x_{d+1}, \dots, x_n = 0$ but it is forced by (ii).

To check the condition that $T C_{\Psi}$ is the nullspace of the differentials, we form the Hessian matrix of$\Psi^{\omega}$,
$$(D \nabla \Psi^{\wt \omega} )|_{C_{\Psi^{\wt \omega}}} = \begin{pmatrix} & y & \omega \\ & &  \\
y & D^2_y \Psi^{\wt \omega} & D^2_{y, \omega} \Psi^{\wt \omega} \\ & & \\
\omega  & D^2_{\omega y} \Psi^{\wt \omega} & D^2_{\omega} \Psi^{\wt \omega} \end{pmatrix}_{ |_{C_{\Psi^{\wt \omega}}}} = \begin{pmatrix} A & B \\ & \\ B^T & D \end{pmatrix}
= \begin{pmatrix} 0 & B \\ & \\ B^T & D \end{pmatrix}.$$
Here, $A$ is $d \times d, B$ is $(n-1) \times d$ and $D$ is $(n-1) \times  (n-1)$. Intrinsically, the Hessian acts on 
$T_q H \oplus T (S^*_q M)$. 

For future reference, we recall \cite{L88} that the signature of an invertible symmetric matrix $M$  as above, with inverse 
$\begin{pmatrix} A' & B' \\ & \\ (B')^T  & D' \end{pmatrix}$ is given by
\begin{equation} \label{SIG2} \rm{sgn} \; M = \rm{sgn} A + \rm{sgn} \; D'. \end{equation}

\noindent{\bf Proof of (ii)} The following Lemma implies (i)-(ii).

\begin{lemma}\label{LAGLEM}  Let $r = \rm{Rank} (D \nabla \Psi^{\wt \omega} )|_{C_{\Psi^{\wt \omega}}}$ and let 
$V = \rm{Ker}\; (D \nabla \Psi^{\wt \omega} )|_{C_{\Psi^{\wt \omega}}}$. 
\begin{enumerate}

\item When $y \not= 0$,  $r = n + d-2$ and $V = \rm{span} \frac{\partial}{\partial r}$, the radial vector in $T_q H$.  \bigskip

\item When $y=0$,  $r = 2 d-2$ and $V=  \rm{span} \frac{\partial}{\partial r} \oplus T S^*_q M$.

\end{enumerate}

\end{lemma}

\begin{proof}
Since $A =0$, the rank of the Hessian is $\; \rm{rank} B + \rm{rank} \begin{pmatrix} B \\ D \end{pmatrix}$.  $D_y \Psi^{\omega} = \pi_H \omega = \wt \omega$ on the critical set and its derivative there   always has rank $d-1$.  Since $\Psi''_{\omega \omega} = II_{S^{n-1}} (y) = - |y| I_{n-1} $ is the second fundamental form of $S^{n-1}$ at point
where $y$ is normal, it equals $-|y|$ times the identity matrix of rank $n-1$. 
When $y=0$ then $D=0$ and the rank equals $2\; \rm{rank} \; B = 2 d-2$. When $y \not= 0$ the rank equals $\rm{rank} \ \; B + \rm{rank} \;D = n + d-2$. \end{proof}
Thus, there is a drop in rank
by $n-d$ when $y =0$ and the phase fails to be clean when $y=0$. But it is of constant rank and is clean for $y \not= 0$.

\noindent{\bf Proof of (iii)} 
\bigskip

The associated
Lagrange map is 
$$\iota_{\Psi}: C_{\Psi} \to T^* S^{d-1}, \;\; \iota_{\Psi}(\wt \omega, y, \omega) = (\wt \omega, d_{\wt \omega} \Psi(\wt \omega, y, \omega)) = (\wt \omega, y - \langle y, \wt \omega \rangle \wt \omega) = (\wt \omega, 0)
\in 0_{T^*S^{d-1}}. $$
The fiber of this map (with $q$ fixed and suppressed) is given by
\begin{equation} \label{FIBER} \iota_{\Psi}^{-1}(\wt \omega, 0) = \{(\wt \omega, y, \omega): \omega = \wt \omega, y = \langle y, \wt \omega \rangle \wt \omega\}
\simeq \R^*,
\end{equation}
since the equation  $y = \langle y, \wt \omega \rangle \wt \omega$ is homogeneous under multiplication of $y$ by $x \in \R$. If we 
denote $v = \frac{y}{|y|}$ then $v = \pm \wt \omega$.

\end{proof}

It will be explained in Section \ref{BDSECT} that $\iota_{\Psi}$ has a blow down singularity over $\{y = 0\}$.

\subsection{\label{STPHASE}Stationary phase}

For the sake of completeness, we recall the stationary phase method (cf. \cite[Volume I]{HoIV}).

\begin{theo} Let $K \subset \R^n$ be compact, let $U$ be an open
neighborhood of $K$, and let $k \in \N$. Let $a \in
C_0^{\infty})K), S \in C^{\infty}(U)$ with $\Im S = 0$.  Assume
$S'(x_0) = 0, \det S''(x_0) \not= 0,
S' \not= 0 $ in $K \backslash
\{x_0\}.$ Then:
$$\begin{array}{l} \int_{\R^n} a(x) e^{i \lambda S(x)} dx = \\ \\
= e^{i \lambda S(x_0)}{\sqrt{\det (\lambda S''(x_0))/2\pi i)}}
\sum_{j < k} \lambda^{-j} L_j a(x_0) \\ \\
+ O(\lambda^{-k} \sum_{|\alpha| \leq 2k} \sup |D^{\alpha} u(x)|).
\end{array}$$
Here, if $g_{x_0}(x) = S(x) - S(x_0) - \langle S''(x_0) (x - x_0),
(x - x_0) \rangle/2$ then
$$L_j a = \sum_{\nu - \mu = j} \sum_{2 \nu \geq 2 \mu} \frac{i^{-j}
2^{-\nu}}{\mu! \nu!}  \langle S''(x_0)^{-1} D, D \rangle^{\nu}
(g_{x_0}^{\mu} a). $$
\end{theo}

\subsection{\label{SPSECT} Stationary phase analysis away when $\hat{\psi}(s)$ vanishes near $s = 0$}

\begin{proof} By Lemma \ref{UNCLEAN2}, stationary phase applies to the subintegral of 
 \eqref{FV} defined by, $$I_2(q, \wt \omega, \lambda) \simeq  \lambda^{n + d -2} \int_{T^*_q H}  
\int_{S^*_q M}  \hat{\psi}( \langle y, \wt \omega \rangle) \hat{\rho}( \langle y, \omega  \rangle)   e^{i \lambda \Psi } \\ \\ \wt Ady   d S(\omega), $$ in the `phase variables' $(\omega,y)$, with $\Psi$ defined in \eqref{PSI},  and with amplitude
$$\wt A = \wt A (r \langle v,\omega \rangle, r \langle v, \wt \omega \rangle,  q, rv ,  \omega,  \wt \omega), $$
where $y= r v$, 
and where  $(q, \wt \omega)$ are regarded as parameters.  By \eqref{CPSI}, the  phase   is stationary if and only if $\omega = \wt \omega =  \pm v = \frac{y}{|y|}$ with $|\det \rm{Hess } \Psi|^{-\half}
= r^{-(n-d)}$. The critical set for fixed $(q, \wt \omega)$ may be identified with $\R_x$; the phase equals zero on the critical set.
        When $y \not=0$ we can use the Schur determinant formula to obtain,
        $$\det \rm{Hess}  =  \det D \det (B^T D^{-1} B)   =  |y|^{n-1} |y|^{-(d-1)} = |y|^{- (n-d)} = r^{- (n-d)}. $$
        The signature of the Hessian at the critical point is given by, 
        $$\rm{sgn} \; \rm{Hess} = 0. $$
        Indeed, it is recalled in \eqref{SIG2} that the signature of the Hessian is the signature of $D'$ where the inverse of the Hessian
        has lower right block $D'$; here we use that the upper left block of the Hessian equals zero. By the Schur complement inverse formula, when $A =0$ the Schur complement is $ M/D = - B D^{-1} B^T$ and 
        $$M^{-1} = \begin{pmatrix} (M/D)^{-1} & - (M/D)^{-1} B D^{-1} \\ & \\
        - D^{-1} B^T (M/D)^{-1} & D^{-1} + D^{-1} B^T (M/D)^{-1} B D^{-1} \end{pmatrix}, $$
        so (again assuming $A =0$)  $$\begin{array}{lll} \rm{sgn} M =  \rm{sgn} \left(D^{-1} + D^{-1} B^T (M/D)^{-1} B D^{-1}\right)  & = & \rm{sgn} 
        \left(D^{-1} + D^{-1} B^T(- B D^{-1} B^T)^{-1} B D^{-1}\right) \\&&\\
        & = &  - \rm{sgn} (I_{n-1} - B^T (B B^T)^{-1} B), \end{array} $$
          since by Lemma \ref{LAGLEM}, $D = - |y| I_{n-1}$, while $B = D_{\omega} \pi_H \omega$. Now $B B^T = I_{n-1}$, concluding the proof.

        Since $$\langle y, \wt \omega \rangle =r \langle v, \omega \rangle = \pm r= x$$ on the critical set, 
  modulo dimensional constants $C_{n,d} $,
        $$\begin{array}{lll} I_2(q, \wt \omega, \lambda)  & \simeq  &    \lambda^{n + d-2}  
 \int_{\omega \in S^{n-1} }  \int_{\R^d }  \hat{\psi}(\langle  y, \wt \omega \rangle) \hat{\rho}(\langle y, \omega \rangle)  \chi_2(y) e^{i \lambda \langle y, \omega - \wt \omega \rangle}  \wt A dy  dS(\omega) \\&& \\ 
& \simeq &  \lambda^{n + d-2}  \lambda^{- \frac{n + d -2}{2}} \;
 \int_{-\infty}^{\infty}  \hat{\rho}(x)  \hat{\psi}(x) \acal_0 (q, \wt \omega, x)  |x|^{- \frac{n-d}{2} + d-1} dx, 
  \\&&\\&& \rm{where}\;\; \acal_0(q, \wt \omega, x) :=  \wt A_0 (x ,x,  q, x \wt \omega, \wt \omega,  \wt \omega).
  \end{array} $$

Assuming that
  $\hat{\rho } = 1$ on $\rm{Supp} \hat{\psi}$ and $\rm{Supp} \hat{\psi} \subset (0, \infty)$, the integral has the same singularity (to leading order) as, 
  $$ C_{n,d} \lambda^{\frac{n + d}{2} -1}  \int_0^{\infty}   \hat{\psi}(s)  s^{- \frac{n-d}{2} } ds. $$

\end{proof}

  This is a preliminary result since we have assumed
the $\hat{\psi}$ vanishes near $0$. The stationary phase expansion is not valid all the way down to $s=0$, as one can easily verify in model
cases such as for Bessel functions (Section \ref{BESSELSECT}). In the next sections we will prove the complete formula for the leading coefficient.  We also calculate the amplitude explicitly by using
a Hadamard parametrix in Section \ref{HADPAR}.
In 
the next section, we use a model integral with a sufficiently accurate Taylor approximation to the phase,  which clarifies the distributional 
coefficient \eqref{acHpsi1}   when $[0, 1] \subset \rm{Supp} \hat{\psi}$.

\section{ \label{MODELSECT} Proof of Theorem \ref{main 1} and \ref{main 2} using a model phase}
In this section, we complete the proof of Theorem \ref{main 1} and Theorem \ref{main 2} by approximating
the phase $\langle y, \omega - \wt \omega \rangle$ by its Taylor expansion up to order $4$. This allows one
to distinguish one variable that causes degeneracy of the stationary phase method. It is possible to apply
the stationary phase method in the remaining variables and then to integrate the result in the distinguished variable. 
\subsection{ Model phase}

In determining properties of the Hessian, and in calculating asymptotics of integrals,  it is convenient to 
  Taylor expand \eqref{Psied} around its critical point $x' = 0$. We define  the `model' phase,
 \begin{equation} \label{PSIMODEL}\begin{array}{lll}  \Psi_{model} (\vec y, \vec x) & =&  \sum_{j=1}^{d-1} y_j x_j - \half  y_d  |x'|^2:  \R^d \times B_1(\R^{n-1}) \to \R \\ &&\\ & = &  \sum_{j=1}^{d-1} y_j x_j - \half  y_d (x_1^2 + \cdots + x_{d-1}^2 +
x_{d+1}^2 + \dots + x_n^2) ). \end{array}  \end{equation} 
   We  view the variables $(y_d, x_d, \dots, x_n)$ as parameters and consider the phase \eqref{Psied} as a function
$\Psi_{model, d-1} $  of $(y_1, \dots, y_{d-1}, x_1, \dots, x_{d-1})$.  The critical point analysis in Lemma \ref{UNCLEAN2} 
applies to the model phase as much as to \eqref{Psied} since they agree modulo terms of order $4$ by the next Lemma.
Since $ \sqrt{1 - |x'|^2} = 1 - \half   |x'|^2 + O( |x'|^4),$ we have
\begin{lemma} \label{MODELPHASELEM} In the above coordinates, when $y \not=0$ the  universal phase 
\eqref{Psied}  has a unique critical point  $\vec x = e_d$ and $y = y_d e_d$, and satisfies,
$$\langle y, \omega - \wt \omega \rangle  = \Psi_{model}(\vec y, \vec x) + O(y_d  |x'|^4). \;\; $$

\end{lemma}




The Hessian of \eqref{Psied} equals that of \eqref{PSIMODEL} at the critical points.

\begin{lemma} \label{NONDEGd-1} For any fixed $(y_d, x_d, x_{d+1}, \dots, x_n)$, the Hessian of \eqref{Psied} (or equivalently,  $\Psi_{model, d-1} $)  at $y' =0= x'$ is non-degenerate in the variables $(y', x') = (y_1, \dots, y_{d-1}, x_1, \dots, x_{d-1})$. Indeed, 
  $$(D^2\Psi_{model})|_{x' = 0, y= y_d e_d} = \begin{pmatrix} & \vec y & \vec x \\ && \\
 \vec y   &  0_{d-1, d-1} &   I_{d-1, d-1}  &    &\\ & &  \\ \vec x &  I_{d-1, d-1} &   - y_d I_{d-1, d-1}  \end{pmatrix}.$$
  Moreover, $\det (D^2\Psi_{model})|_{x' = 0, y= y_d e_d} =1$, and its inverse is given by,
$$(D^2\Psi_{model})^{-1}|_{x' = 0, y= y_d e_d}  = 
 \begin{pmatrix} y_d I_{d-1, d-1}   &  I_{d-1, d-1} \\ & \\  I_{d-1, d-1} &  0_{d-1, d-1}   \end{pmatrix}.$$
     \end{lemma}

  In particular, we note that the determinant and inverse  of this Hessian are uniformly bounded, i.e. do not blow up
  when $y_d \to 0$. 
However, when  $y_d = 0$ the critical point equations do not imply that
$x_{d+1} = \cdots = x_n =0$.  In invariant terms, $y_d = 0$ corresponds to $\Delta_{H \times H}$ and the coordinates $(x_{d+1}, \dots, x_n)$
run over the fiber of $N^* \Delta_{H \times H}$.

\begin{proof} The Hessian has the form   $$(D^2\Psi_{model})|_{x' = 0, y= y_d e_d}  = \begin{pmatrix} & y & x'\\ & &  \\
y & D^2_y \Psi_{model} & D^2_{y, x'} \Psi_{model}\\ & & \\
x' & D^2_{x' y} \Psi_{model} & D^2_{x'x'} \Psi_{model} \end{pmatrix}_{ |_{C_{\Psi^{\wt \omega}}}} = \begin{pmatrix} 0 & B \\ & \\ B^T & D \end{pmatrix}
,$$
with $A = 0$ and with,
$$\begin{array}{ll} B_{kj} =  D^2_{y_k x_j}  (\sum_{j=1}^{d-1} y_j x_j - \half  y_d  |x'|^2) = \delta_{jk} = I_{d-1,d-1} & \;\; k =1, \dots, d-1, j = 1, \dots, d-1, 
\\ & \\
 D_{jk} : = D^2_{x_j x_k}    (\sum_{j=1}^{d-1} y_j x_j - \half  y_d  |x'|^2) = - y_d \delta_{jk}, & j,k = 1, \dots, d-1,
\end{array}$$
proving that the Hessian has the stated form. Since we can multiply the top row block by $y_d$ and add it to the bottom row block without
changing the rank, the matrix has full rank.

We calculate the determinant  Schur determinant formula by interchanging the two columns and using the
Schur formula $\det M = \det D \det (A - B D^{-1} C)$.

\end{proof}


 \subsection{\label{DEGPHINTRO} Asymptotics of the model integral}

In this section, we  drop the factors $\wt A e^{i \lambda R_4}$ from the amplitude.\eqref{N1q} and \eqref{udef}, and study
the model oscillatory integral,
\begin{equation} \label{MODELOSCINT} I_{model} (\lambda): = \int_{B_1(\R^{n-1}) } \int_{\R^d} \chi(y) \hat{\psi}(y_d) e^{i \lambda \left( \sum_{j=1}^{d-1} y_j x_j  - \half   y_d (x_1^2 + \cdots + x_{d-1}^2 +
x_{d+1}^2 + \dots + x_n^2) \right)}  d \vec y d \vec x. \end{equation}
In the next section, we restore the factors and explain their role in the final answer.

By Lemma \ref{NONDEGd-1}, we can remove the variables $(y_1, \dots, y_{d-1}, x_1 \dots, x_{d-1})$ by applying  stationary phase to the sub-integral,  $$I(\lambda, y_d, x_{d+1}, \dots, x_n): = \int_{\R^{d-1}}\int_{\R^{d-1}}  \chi(y', y_d)  e^{i \lambda \left( \sum_{j=1}^{d-1} y_j x_j - \half y_d  (x_1^2 + \cdots + x_{d-1}^2 +
x_{d+1}^2 + \dots + x_n^2) \right)}  d \vec y' d \vec x''. $$
As in Section \ref{STPHASE}, there exists a complete asymptotic expansion with leading term,
 $$I(\lambda, y_d, x_{d+1}, \dots, x_n) \simeq \lambda^{- \frac{d-1 + d-1}{2}}  \chi(0, y_d)e^{- \half i \lambda 
y_d (x_{d+1}^2 + \dots + x_n^2) }. $$
The higher order terms in $\lambda^{-1}$ involve the inverse Hessian derivatives of the amplitude multiplied by 
$e^{i \lambda R_3}$ where $R_3$ is the third and higher order terms of the phase. As noted below Lemma \ref{NONDEGd-1}, the inverse Hessian operators have smooth coefficients and so the remainders are as stated
in Section \ref{STPHASE}.

After applying stationary phase,  the integral is reduced to a series of which the leading term is,
$$\lambda^{- \frac{d-1 + d-1}{2}}  \int_{\R} \int_{B_1(\R^{n-d})} \chi(0, y_d) \hat{\psi}(y_d) e^{- \half i \lambda y_d (x_{d+1}^2 + \dots + x_n^2) }  dy_d dx_{d+1} \cdots d_{x_d}. $$
We put the integral in polar coordinates with radial variable $R^2 = (x_{d+1}^2 + \dots + x_n^2) $  to get,
$$I(\lambda, y_d, x_{d+1}, \dots, x_n) \simeq  \lambda^{- \frac{d-1 + d-1}{2}}  \int_{\R} \int_0^1 \chi(0, y_d) \hat{\psi}(y_d) e^{- \half i \lambda y_d 
R^2} dy_d  R^{n-d-1} dR,$$
where we obtain the new amplitude $\wt A_1$ from the stationary phase expansion and integration over the unit
sphere in $\R^{n-d}$.
Let us write $\rho = \half R^2$ to get 
$$I(\lambda, y_d, x_{d+1}, \dots, x_n) \simeq    \lambda^{- \frac{d-1 + d-1}{2}}  \int_{\R} \int_0^{\half} \chi(0, y_d) \hat{\psi}(y_d) e^{- \half i \lambda y_d \rho}  dy_d \rho^{\frac{n-d-1}{2}} \rho^{-\half} d\rho.$$
Note that the integrand is in $L^1$ for any $d \leq n -1$.
Thus, the  model oscillatory integral reduces to
$$\begin{array}{lll}  \int_{\R}  \int_0^{\half}   \hat{\psi}(y_d)  e^{i \lambda y_d  R^2} R^{n-d-1}  dR d y_d & = &  \int_0^{\half} \psi( \lambda  R^2)  R^{n-d-1}  dR \\ &&\\ &&= \lambda^{- \frac{n-d}{2} } \int_0^{\lambda} \psi( \rho) \rho^{\frac{n-d-2}{2}} d\rho\\&&\\
& = & \lambda^{- \frac{n-d}{2} }   \int_0^{\infty} \psi( \rho) \rho^{\frac{n-d-2}{2}} d\rho + \ocal(\lambda^{-\infty}). \end{array}$$

We then rewrite the answer in terms of the Fourier transform \eqref{GS},\begin{equation} \label{FT1} \begin{array}{lll}   \int_{\R} \psi( \rho) \rho_+^{\frac{n-d-2}{2}} d\rho & = &  \int_{\R}  \hat{\psi}(s) \fcal^{ *}_{\rho \to s} \rho_+^{\frac{n-d-2}{2}}  ds
\\&&\\ &=& i e^{i \lambda \pi/2} \Gamma(\frac{n-d-2}{2}+ 1)  \int_{\R} \hat{\psi}(s) (s + i 0)^{- \frac{n-d}{2}} ds. \end{array} \end{equation}

Multiplying by the factor $\lambda^{n + d -2} \lambda^{-d +1}  $ from the prior calculations, the model integral becomes
\begin{equation} \label{MODELFT}\begin{array}{lll}  I_{model} (\lambda) & \sim &C_{n,d}  \lambda^{n + d -2} \lambda^{-d +1}  \lambda^{- \frac{n-d}{2} } \int_{\R} \hat{\psi}(s) (s + i 0)^{- \frac{n-d}{2}}ds\\ &&\\  &= &C_{n,d} \lambda^{ \frac{n+d}{2} -1 } \int_{\R} \hat{\psi}(s) (s + i 0)^{- \frac{n-d}{2}} ds.\end{array} \end{equation}

 \subsection{\label{FULLAMP} Completion of the proof of Theorem  \ref{main 2} }
 The purpose
of the above calculation was to exhibit a simple model which gives the same type of leading coefficient. We now complete the proof
of the first (short-time) statement of Theorem \ref{main 1} and Theorem \ref{main 2} by including  the additional amplitude factors $e^{i \lambda R_4}\wt A$
of  Lemma \ref{MODELPHASELEM} 
in the integrand.

\begin{proof}

We  repeat the analysis in Section \ref{DEGPHINTRO} but replacing the amplitude
$\chi(y) \hat{\psi_d} $ in \eqref{MODELOSCINT} by the full amplitude
$$\chi(y)  \hat{\psi}(y_d) \wt A (x_1, \dots, x_{d-1}, y_1, \dots, y_{d_1},  y_d, x_{d+1}, \dots, x_n) e^{i \lambda R_4(x_1, \dots, x_{d-1}, y_1, \dots, y_{d_1}, , y_d, x_{d+1}, \dots, x_n)}. $$
As in the proof of the stationary phase method in \cite[Volume I]{HoIV}, the factor $e^{i \lambda R_4}\wt A$ can be absorbed
into the amplitude and then produces the expansion reviewed in Section \ref{STPHASE}.  The stationary phase procedure applies with this additional factor as in the  model case since the
phase is the same as the model case and  since, by Lemma \ref{NONDEGd-1},  the inverse Hessian derivatives are smooth. 

Stationary phase in the variables $(x_1, \dots, x_{d-1}, y_1, \dots, y_{d-1})$ localizes the integrand to 
 $(x_1, \dots, x_{d-1}, y_1, \dots, y_{d-1}) = \vec 0.$  The new part of the integrand is,
$$ \wt A (\vec 0, \vec 0,   y_d, x_{d+1}, \dots, x_n) e^{i \lambda R_4(\vec 0, \vec 0 , y_d, x_{d+1}, \dots, x_n)}. $$
  We then use polar coordinates $(x_{d+1}, \dots, x_n) = R \omega'$ and again use that the model phase is
  $e^{- \half i \lambda y_d R^2}.$
  
 The other factor in the amplitude is  
$$ \wt A (\vec 0,  \vec 0,   y_d, x_{d}, \dots, x_n) $$
We set
$$\acal (y_d) = \int_{\omega' }  \wt A (\vec 0,  \vec 0,   y_d, x_{d}, \dots, x_n) d \mu.$$  
  
 The resulting integrals have the form,
 $$\lambda^{- \frac{d-1 + d-1}{2}}  \int_{\R} \int_{B_1(\R^{n-d})} \chi(0, y_d) \hat{\psi}(y_d) \acal(y_d) e^{- \half i \lambda y_d \sqrt{1 - R^2} -1)  } R^{n-d-1}  dy_d d R . $$
 Next we integrate in $y_d$ to get $$\fcal_{y_d \to \eta}  \chi(0, y_d) \hat{\psi}(y_d) \acal(y_d) |_{\eta =   (\half  \lambda  \sqrt{1 - R^2} -1)}.   $$
 We then get the explicit integral above, and inverse Fourier transform to get 
 $$\int_{\R} \left(\chi(0, y_d) \hat{\psi}(y_d) \acal(y_d)\right) (y_d + i 0)^{- \frac{n-d}{2}} dy_d. $$

 Modulo calculating $\acal(y_d)$, this gives all the details of  the leading order term  in Theorem \ref{main 1} and Theorem \ref{main 2}.  The calculation 
 of the density is given in Section \ref{HADPAR}.
 \end{proof}

\subsection{\label{SINGINTRO} Apriori properties of  $a^0_1(H, \psi)$. }

In view of the complications in computing the leading coefficient using a hybrid stationary phase 
and Fourier inversion method, we use an indirect argument
to prove that the regularization $(s + i 0)^{-\frac{n-d}{2}}$ in \eqref{acHpsi1} is the correct regularization. 

\begin{lemma} \label{a10LEM}  The functional $\psi \to a_1^0 (H, \psi) $ in \eqref{acHpsi1} 
 is a positive measure on $\R$
which  is supported on $\R_+$. 
\end{lemma}

\begin{proof} Since it is now proved that \eqref{cpsi} has an asymptotic expansion of order $ \lambda^{\frac{n+d}{2}} $,
the leading coefficient is given by,
$$a_1^0 (H, \psi) := \lim_{\lambda \to \infty} \lambda^{-\frac{n+d}{2}} N_{\psi, H}^{1} (\lambda). $$ 
Since $N_{\psi, H}^{1} (\lambda) \geq 0$ when $\psi \geq 0$, also
$a_1^0 (H, \psi) \geq 0 $ if $\psi \geq 0$. To prove that  it is supported on $\R_+$ we note that by the convention in \eqref{cpsi}, the differences of eigenvalues
are  ordered as $\lambda_j - \mu_k$, and from the fact that the Fourier coefficients \eqref{FC} are negligible for $\mu_k \geq \lambda_j +
\epsilon$ we see that the positive measure  $ a_1^0(H, \psi)  $ is supported on the positive reals.

In \eqref{MODELFT} we get the inverse Fourier transform  of this measure. It is determined on $(0, \infty)$ by the stationary phase analysis above, which 
shows that it agrees with the formula \eqref{acHpsi1}  for $\psi$ such that $\hat{\psi} = 0$ in some interval around 
$0$. This leaves two points unclear: (i) how the Fourier transform is regularized at $s=0$; (ii) whether there exists
a component of the Fourier transform supported at $s=0$.
Regarding point (i), there is apriori only one regularization $s^{-\frac{n+d}{2}}$ for $s > 0$ whose  Fourier transform is a temperate  positive measure  supported on $\R_+$. Indeed, all regularizations must agree on $(0, \infty)$ and their
differences must be supported at $s=0$. Regarding point (ii) the analysis of the model phase shows that there is no component supported at $s=0$.
\end{proof}


\subsection{\label{BESSELSECT} Relation to Bessel integrals}

The  singularity at $s=0$ is therefore universal. The  above model integral with $\wt A = 1$ is essentially the `double Bessel integral'. To explain this, we recall the well-known formula,    \begin{equation} \label{FTSn}|\lambda \xi|^{- \frac{n-2}{2}} J_{\frac{n-2}{2}} (2 \pi \lambda |\xi|) = \int_{S^{n-1}} e^{2 \pi i \lambda \langle \xi, \omega \rangle} dS(\omega), \;\;\; (\xi \in \R^n). \end{equation}
   Stationary phase asymptotics apply when $|\lambda \xi | \to \infty$ but do not apply when $|\lambda \xi| \leq M$ for some $M > 0$.  The rapid decay of the  integral (as $\lambda \to \infty$),
$$\begin{array}{l} \int_{\R^n} \hat{\psi}(y) \left(\int_{S^{n-1}} e^{i \lambda \langle y, \omega \rangle } d \omega \right) dy= \int_{S^{n-1}} \psi(\lambda \omega) d \omega = \psi(\lambda), \;\; \rm{if}\;\; \psi(y) = \psi(|y|),
\end{array}$$
for a radial function $\psi$ on $\R^n$ is obvious from the fact that there are no critical points
of $y \to \langle y, \omega \rangle$,  but    becomes opaque if one tries to first  apply
stationary phase in $\omega$, or to express it in terms of Bessel functions,
$$\begin{array}{lll} \int_{\R^n} \hat{\psi}(y) \left(\int_{S^{n-1}} e^{i \lambda \langle y, \omega \rangle } d \omega \right) dy
& = & \int_0^{\infty} \hat{\psi}(r) (\lambda r)^{-\frac{n-2}{2}} J_{\frac{n-2}{2}}(\lambda r) r^{n-1} dr. 
\end{array}$$
The well-known (stationary phase)  asymptotics of $  (\lambda r)^{-\frac{n-2}{2}}  J_{\frac{n-2}{2}}(\lambda r)$ only
apply when $\lambda r \to \infty$. 

    The singularity
   at $s=0$ is essentially of this type.  To be more exact, it arises in the Euclidean case as  the ``double-Bessel'' function,
    \begin{equation} \label{DB} \begin{array}{lll} J^1_{\R^d, \R^n} (\lambda, y, \psi): & = &  \int_{\Ss^{n-1}}  \int_{\Ss^{d-1}} \hat{\psi} ( \langle  y,
   \wt \omega \rangle)  e^{i \lambda \langle y,  \pi_{\R^d} \omega - \wt \omega \rangle}  dS_n (\omega) dS_d(\wt \omega), \;\; y \in \R^d, \\ &&\\
   & = &  (r \lambda)^{- \frac{n-2}{2}} J_{\frac{n-2}{2}}(\lambda r)
 (r \lambda)^{- \frac{d-2}{2}} J_{\frac{d-2}{2}}(\lambda r), \;\; \rm{if}\;\; \hat{\psi} =1, \;\; r = |y|.  \end{array}\end{equation}
The model phase \eqref{PSIMODEL} approximates the phase of this integral. The parameter $s$ is $|y|$. The asymptotics of \eqref{DB} for
$\lambda R \leq C$ are obtained by Taylor expansion and are clearly of larger order in $\lambda$ than the stationary phase asymptotics
for $(\lambda r) \to \infty$.

\subsection{\label{COMPSECT} Comparison to the case $c < 1$}
The asymptotics above involve test functions $\rho$ such that the support of $\hat{\rho}$ is an interval $[-\epsilon,
\epsilon]$ which contains no periods $T \not=0$ of the $G^t_M$. We now compare the results for such test
functions in the case $c < 1$ and $c=1$. Later, we will compare results for general test functions $\hat{\rho}$.
In the case $c < 1$ of \cite{WXZ21} the phase $\langle y, \omega + \wt \omega \rangle$ of Proposition \ref{N1PARAM} and of \eqref{PSI} 
gets replaced by $\Psi_c (\wt \omega; y, \omega) = \langle y,  \omega - c \wt \omega \rangle$. Its critical set is given by $C_{\Psi_c} = \{ \pi_H \omega = c \wt \omega, y =  \langle y, \omega \rangle \pi_H \omega$, rather than the set \eqref{CPSI}. The main difference is that the critical point
equations do not constrain $\pi_H^{\perp} \omega \in N^* H$ when $c < 1$  except in its norm $|\pi_H^{\perp} \omega | = \sqrt{1 - c^2}$; when $c=1$,
the normal component vanishes.  Hence, 
$\dim C_{\Psi_c} = \dim C_{\Psi} + (n-d-1)$, and   $\dim C_{\Psi_c} =  d-1 + (n-d - 1) =  n -2$.  
When $H$ is a hypersurface, 
 $\dim N^c_q H = 0$ and so $\dim C_{\Psi}  = \dim C_{\Psi_c} +1$. When $\dim H = n-2$ then $\dim N_q^c H =1$
 and $\dim C_{\Psi}  = \dim C_{\Psi_c}$. Otherwise, $\dim C_{\Psi}  > \dim C_{\Psi_c}$.
This difference in dimensions is responsible for the change in order of growth of $N^c_{\psi, H}(\lambda)$
from $\lambda^{n-1}$ when $c < 1$ \cite[Theorem 1.1]{WXZ21}  to $\lambda^{\frac{n+d}{2} }$ in Theorem \ref{main 1}. Moreover, the degeneracy of the\eqref{PSI} when $y = 0$ does not occur when $c < 1$ since in the coordinates above, $\omega = (x_1, \dots, x_d, 
x_{d+1}, \dots, x_n)$ with $x_1^2 + \cdots + x_d^2 = c^2, x_{d+1}^2 + \cdots + x_n^2 = 1 -c^2$ and 
$\langle y,  \omega -  c \wt \omega \rangle =   \sum_{j=1}^{d-1} y_j x_j +  y_d (x_d - c)  $, and the  analogue of the local model \eqref{PSIMODEL} for $c=1$  is 
$\Psi_{c, model} (\vec y, \vec x) =  \sum_{j=1}^{d-1}  y_j x_j  - y_d (\sqrt{c^2 - (x_1^2 + \cdots + x_{d-1}^2)} - c).  $ The critical points are given by,
$x_j = y_j =  0, (j=1, \dots, d-1), x_d = c $ but $y_d$ and  $(x_{d+1}, \dots, x_n)$ are unconstrained by the critical point equation except that the norm 
of  $(x_{d+1}, \dots, x_n)$ is
$\sqrt{1 -c^2}$. The analogue of the  equation \eqref{GEO1} when $c < 1$ is $G_H^{-s} \pi_H G^{t + c s} (q, \xi)
= \pi_H (q, \xi)$, and there are  solutions along the diagonal only when $s=t= 0$, since otherwise  the bi-angle must make a non-zero angle with $H$. 
Since $(x_{d+1}, \dots, x_n)$ lie in the critical set,  the normal  Hessian has no $(x_{d+1}, \dots, x_n)$-block as in Section \ref{SPSECT} and therefore
does not acquire the factor $|y|^{-(n-d)}$.

\section{\label{HADPAR} Calculation of the amplitude in     \eqref{acHpsi1} using the Hadamard parametrix} 

In this section, we use Hadamard parametrices for $U_M(t, x, y)$ resp. $U_H(s, x, y)$ to  give a  formula \eqref{ALT} for $N^1_{\psi, \rho, H}$. The Hadamard parametrices are simple and explicit enough to identify the geometric
invariants in \eqref{acHpsi1}. On the other hand, the calculations in polar coordinates become singular at $r=0$ and the Hadamard parametrix does not
seem to give a simple approach to the singularity at $s=0$.

The Hadamard parametrix uses the phase $\sigma (t^2 - r^2(x, y) $ where $\sigma > 0$ and where $r(x, y)$ is the Riemannian distance
between $x, y$. The distance squared $r^2(x,y)$ is smooth in a neighborhood of the diagonal $x=y$ but is not smooth when $y$ is a cut point
to $x$. Hence we need to cutoff the parametrix using a cutoff $\chi(x,y)$ sufficiently near the diagonal so that $r^2$ is smooth. We absorb the
cutoff into the amplitude and suppress it from the notation.  The neighborhood
is the union of the same balls $B_x(\epsilon)$ in the definition of the H\"ormander parametrix. We then denote the volume density in geodesic coordinates centered
 at $x$ by $dV_g = \Theta(x, y) dy$ (see Section \ref{HADPAR} and  \cite{Be} for  background on $\Theta(x,y)$).

The Hadamard parametrices for the half-wave group $U_M(t, x, y)$ of $M, $ resp. $U_H(t, x, y)$ of $H$ are given (at least
for $t \geq 0$)  by
$$\left\{ \begin{array}{l} U_M(t, x, y) = \int_0^{\infty} e^{i \sigma (t^2 - r_M^2)} A_M(t, x, y, \sigma) \sigma_+^{\frac{n-1}{2}} d \sigma, \\ \\
U_H(t, x, y) = \int_0^{\infty} e^{i \sigma (t^2 - r_H^2)} A_H(t, x, y, \sigma) \sigma_+^{\frac{d-1}{2}} d \sigma. \end{array} \right.,$$
where $r_M$ resp. $r_H$ are the distance functions of $M$ resp. $H$. Since $H$ is totally geodesic, $r_H = r_M$ on $H \times H$.   The amplitudes $A_M, $ resp. $A_H$ are divisible by $t$  have  the asymptotic symbol expansions as $\sigma \to \infty$,
$$\left\{ \begin{array}{l} A_M(t, x, y, \sigma) \simeq t \sum_{j=0}^{\infty} U_j^M(x, y) \sigma^{-j}, \\ \\
A_H(t, x, y, \sigma) \simeq t \sum_{j=0}^{\infty} U_j^H(x, y) \sigma^{-j }. \end{array} \right. $$

 The original Hadamard-Riesz parametrix constructions did not express these wave kernels as oscillatory integrals. Rather, they
expressed them as infinite series in what are now called Riesz kernels. For instance,
\begin{equation} \label{COS} \cos t\sqrt{\Delta}(x,y) =
 C_0 |t| \sum_{j=0}^{\infty}(-1)^j U_j(x,y)\frac{
(r^2-t^2)_{-}^{j-\frac{n -3}{2} - 2}}{4^j \Gamma(j - \frac{n-3}{2}
- 1)}
 \;\;\;\mbox{mod}\;\;C^{\infty}. \end{equation}
 The oscillatory integral formula in \cite{Be} is obtained by using  the Fourier transform formula \eqref{GS}.

For small $t$, and for any Riemannian manifold $(X, g)$  of dimension $n$, $U_X(t, x, y)$ is to leading order similar to the Euclidean
case  of $\R^n$, whose  half-wave kernel is given by  \begin{equation} \label{UFORM} U_{\R^n}(t, x, y) = C_n'\; \frac{ t}{ ( (t + i 0)^2 - r^2)^{ \frac{n +1}{2}}} = C_n\;\lim_{\tau \to 0} \frac{ i t}{ ( (t + i \tau))^2 -  r(x, y)^2  )^{ \frac{n + 1}{2}}}, \end{equation}
for  constants $C_n', C_n$ depending only on the dimension. A good way to understand the relevant regularizations  is that  $U(t)  = e^{it \sqrt{-\Delta}}$ has a holomorphic extension to $U(t + i \tau)$ for $\tau > 0$ because
$\sqrt{-\Delta}$ is a positive operator. Note that \eqref{UFORM} is Poisson kernel at imaginary time and that $U(t, x, y)$
does not have finite propagation speed, hence does not satisfy Huyghen's principle and is therefore different from 
the cosine kernel \eqref{COS}.

We briefly review the construction of the amplitude by a series of transport equations. Our main references are \cite{Be,Zel12}. The following complications should be kept in mind.
\begin{itemize}

\item As is carefully explained in \cite{Be}, the parametrices for $\cos t \sqrt{-\Delta}$ and for $\frac{\sin t \sqrt{-\Delta}}{\sqrt{-\Delta}}$ are first
derived for $t > 0$ and then extended to all $t$ using the even/odd property of cosine/sine. The coefficient $|t|$ in the formula \eqref{COS} 
seems singular but of course the kernel is analytic in $t$. This only indicates that the parametrix (but not the wave kernel)  is singular at $t=0$ and $t=0$
due to polar coordinate singularities; this is discussed further below.  \bigskip 

\item The regularization procedure of Riesz (by analytic continuation of Riesz kernels) introduces many constants. One arises from the Fourier transform in \eqref{GS}.   More are introduced
by requiring that $U(0, x, y) = \delta_x(y)$.  
In \cite[(11)]{Be}, the  coefficients $U_k$ are converted to coefficients $u_k$ on a manifold of dimension $d$   by the formula $U_k(x, y) = C_0 e^{- i (\frac{d-1}{2} + k) \pi/2} 4^{-k} u_k(x, y). $ The product of these constants and others arising in stationary phase constitute the constant $C_{n,d}$
in Theorem \ref{main 1} and are ultimately responsible for the shape of \eqref{acHpsi1}. To avoid a lengthy (and futile) chasing of constants, we
note that the final coefficients are universal and can be calculated on $\R^n$ or $S^n$ (see Section \ref{Snexact} for the exact formulae on spheres).

\end{itemize}
\bigskip

We now recall some of the details of the Hadamard parametrix construction for $U_M(t) = \exp i t\sqrt{-\Delta}_M $ as an oscillatory integral
of the form,
\begin{equation} U_M(t, x, y) = C_{n,d} \; t \;  \int_{0}^{\infty} e^{i \sigma (r^2-t^2)} \sigma_{+}^{\frac{n-1}{2} - j}
\sum_{j=0}^{\infty} W_j(x,y) \sigma_{+}^{ - j}
d\theta
 \;\;\;\mbox{mod}\;\;C^{\infty}  \end{equation}
where $\sigma^s_{+}$ is the regularization of the distribution $\sigma^s$ for $s$ with negative real part (see  \cite{Be}).  

Hadamard himself did not treat $U_M(t)$ but rather the cosine propagator $\cos t \sqrt{-\Delta} t$ and the sine propagator
$(-\Delta^{-\half}) \sin t \sqrt{-\Delta} $, both of which are formally Taylor series in $\Delta$. To obtain the Hadamard parametrix
for $U_M(t)$ one may apply $\sqrt{-\Delta} $ to $(-\Delta^{-\half}) \sin t \sqrt{-\Delta} $ to obtain its imaginary part and add it to
$\cos t \sqrt{-\Delta}$. 

The 
 Hadamard-Riesz coefficients $W_j$ for  $(-\Delta^{-\half}) \sin t \sqrt{-\Delta} $  are  determined
inductively by the transport equations, 
\begin{equation}\begin{array}{l}
 \frac{\Theta'}{2 \Theta} W_0 + \frac{\partial W_0}{\partial r} = 0\\ \\
4 i r(x,y) \{(\frac{k+1}{r(x,y)} +  \frac{\Theta'}{2 \Theta})
W_{k+1} + \frac{\partial W_{k + 1}}{\partial r}\} = \Delta_y W_k.
\end{array}\end{equation} The solutions are given by:
\begin{equation}\label{HD} \begin{array}{l} W_0(x,y) = \Theta^{-\half}(x,y) \\ \\
W_{j+1}(x,y) =  \Theta^{-\half}(x,y) \int_0^1 s^k \Theta(x,
x_s)^{\half} \Delta_2 W_j(x, x_s) ds
\end{array} \end{equation}
where $x_s$ is the geodesic from $x$ to $y$ parametrized
proportionately to arc-length and where $\Delta_2$ operates in the
second variable. As above, 
 $\Theta(x,y)$ is the volume density in geodesic normal coordinates based at $x$, $dV_g = \Theta(x, y) dy. $ If we change
to geodesic  polar coordinates $(r, \omega)$, we get $dV_g = J(r, \omega) dr d \omega$ where 
where $J$ is defined by \eqref{VOLDEN} 
 In particular, we see that apart from an overall factor of $t$, the amplitudes $A_M, $ resp. $A_H$, above are independent of $t$. The parametrix
 for $\cos t \sqrt{-\Delta}$ is obtained by differentating that for $(-\Delta^{-\half}) \sin t \sqrt{-\Delta} $  in $t$ and the parametrix for $U_M(t)$ is 
 of course obtainined by applying $\sqrt{-\Delta}$ to the latter and adding the real part.  It is straightforward to see that the leading order amplitude
 remains $J$.  The details are given in \cite{Zel12} and will not be repeated here.

Since $U_H(-s) = U_H(s)^*$, the Hadamard parametrix for $U_H(-s)$ has the form, 
\begin{equation} \label{CXCONJ} U_H(-s, x, y) = \overline{U_H(s, y, x)} = \int_0^{\infty} e^{- i \sigma (s^2 - r_H^2)} \overline{A_H(s, y, x, \sigma) }d \sigma. \end{equation}
If we use the Hadamard parametrices and recall \eqref{CXCONJ},  \eqref{NpsirhoDEF}   becomes,

\begin{equation} \label{ALT} \begin{array}{l} 
N^1_{\psi, \rho, H}(\lambda) 
 =
 \lambda^{\frac{n +d}{2}+1}  \int_{H \times H} \int_{\R}\int_{\R} \int_0^{\infty} \int_0^{\infty}
\hat{\psi}(s) \hat{\rho}(t)  
 e^{-it \lambda} 
  e^{i\lambda \left( \sigma_1 (t+s)^2 - \sigma_2 s^2 - (\sigma_1 - \sigma_2) r_H ^2(q,q'))\right)}\\ \\ 
s (t + s) A_M(q, q', \lambda \sigma_1)    A_H(q, q',\lambda  \sigma_2)  \sigma_{1 +}^{\frac{n-1}{2}}   \sigma_{2 +}^{\frac{d-1}{2}} ds dt d \sigma_1 d\sigma_2 dV_H(q) dV_H(q').

\end{array}  \end{equation}
  Again the amplitudes$A_H, A_M$ are defined above and incorporate the
notational conventions for the constants. This is the starting point for the stationary phase analysis in the
next section \ref{HDPFSECT}.

We remark that the     Hadamard parametrix is singulart  at $t=0$ on the diagonal, which motivated the
  the use of  the H\"ormander parametrix method in Section \ref{HORPAR}. 
The singularity arises because the phase of the Hadamard parametrix is expressed in geodesic polar coordinates which become singular on the diagonal. As a result, the canonical relation,
$$C = \{(t, 2t \sigma, x, - \sigma d_x r^2, y, \sigma d_y r^2): t^2 = r^2\} \subset T^*(\R \times M \times M),$$
generated by the phase $\sigma (t^2 - r^2)$ of the half-wave kernel  has an apparent singularity (a $0$ in the wave front relation) when $r = t = 0$, whereas
in fact when $t=0$ it is the graph of the identity map. I.e. the co-normals to the distance spheres collapse to the unit cotangent space at the origin.

\subsection{\label{Snexact} Exact calculations for $M = \Ss^n$. } Since $U_M(t)$ is more complicated than  $(-\Delta^{-\half}) \sin t \sqrt{-\Delta} $  
or  $\cos  t \sqrt{-\Delta} $, we illustrate the result in the case of the standard sphere.  
We define the wave kernel  $U_{\Ss^n}(t) = \exp i t A$ in terms of  the degree  operator $A = \sqrt{- \Delta + \frac{(n-1)^2}{4}} - \frac{n-1}{2}$, which has eigenvalue $N$ in the
space of $N$th degree spherical harmonics. Then as calculated in \cite{Tay},
$$U_{\Ss^n}(t, x, y) = \frac{2 i \sin t}{|S^{n-1}|  } \lim_{\epsilon \to 0}  (2 \ \cos  (t + i \epsilon) - 2 \cos r(x,y))^{-\frac{n+1}{2}}. $$
As above, $U_{\Ss^n}(t, x, y) $ has a holomorphic
extension in $t$ to the upper half-plane and on the real $t$ axis is the boundary value of this holomorphic function.

\subsection{\label{HDPFSECT} Determination of the amplitude in Theorem \ref{main 1} by the Hadamard parametrix method}
We employ the Hadamard parametrix to give a simple determination of the amplitude in the leading coefficint of Theorem \ref{main 1}. As mentioned above, we denote any  dimensional constant by  $C_{n,d}$; 
it is understood that the constant may change in each usage. 

In what follows, we
 restrict the stationary phase analysis to the regimes $t + s > 0$ for $M$ and $s > 0$ for $H$ or
 $t + s < 0$ for $M$ and $s < 0$ for $H$ and show that there are no  points in the canonical relation $C$ in  the complementary cases.
We therefore break up the proof into the cases $t + s > 0, s > 0$ and $t + s < 0, s < 0$ and explain (in more detail than above) why the complementary
cases $t + s < 0, s > 0$ and $t + s > 0, s < 0$ do not contribute to the asymptotics. The universal constants arising
in the two cases may be different and we denote them by $C_{n,d}^+$ and $C_{n,d}^-$.

\subsection{Critical point analysis for $t + s > 0, s > 0$}

\begin{proof}

 We rewrite the integral  \eqref{ALT} in geodesic polar coordinates centered at $q \in H$. 
\begin{equation} \label{INT} \begin{array}{l}
N^1_{\psi, \rho, H}(\lambda) =   \lambda^{\frac{n +d}{2}+1}   \int_{S^* H}  \int_{\R} \hat{\psi} (s) I_{\rho} (q, s, \omega, \lambda)  dV_H(q) d S(\omega), 
\end{array}
\end{equation} with phase,
\begin{equation} \label{phaseDEF} \begin{array}{lll} \Psi & = &  -t +  \left( \sigma_1 (t+s)^2 - \sigma_2 s^2 - (\sigma_1 - \sigma_2) r_H ^2(q,q'))\right) \\&&\\&& = -t  +  \sigma_1[ (t + s)^2 - r^2]  - \sigma_2 [s^2 - r^2], \end{array} \end{equation}
where \begin{equation} \label{IDEF}\begin{array}{l}
I_{\rho}(q, s, \omega, \lambda): = 
 \int_0^{\infty} \int_{\R} \int_0^{\infty} \int_0^{\infty}
\hat{\rho}(t)  
 e^{i  \lambda \Psi} \wt{A}(s,t,  r, \lambda \sigma_1, \lambda \sigma_2)  r^{d-1}  \sigma_1^{\frac{n-1}{2}}   \sigma_2^{\frac{d-1}{2}}  dt dr d \sigma_1 d\sigma_2,
\end{array}  \end{equation}
where $\sigma_j^z$ is regularized by $(\sigma_j)_+^z$ as described in the previous section, and where
$$\wt{A} (s, t, r, \lambda \sigma_1,  \lambda \sigma_2): = (s +t) s
\int_{S^* H }  A_{M \times H}  (r, \omega, \lambda \sigma_1,\lambda \sigma_2) dV_H(q) d S(\omega). $$
Here, 
$$A_{M \times H} (q,q', \lambda \sigma_1, \lambda \sigma_2): =  A_M (q, q', \lambda \sigma_1) \overline{A_H(q,q', \lambda \sigma_2)}, $$
is a semi-classical symbol with principal term as $\lambda \to \infty$,
\begin{equation} \label{A0} A^0_{M \times H} = \Theta_M^{-\half} (q,q') \Theta_H^{-\half}(q,q') =  \Theta_H^{-1}(q,q')  . \end{equation}
We treat $q$ as a parameter and write $q' = \exp_q r \omega$ with $\omega \in T_qH$ (identified with $S^{d-1}$) 
and $r = r_H(q,q')$. As above, we note that although $ \wt{A}(s,t,  r, \sigma_1, \sigma_2)$ apriori depends on $s,t$,
in fact it is independent of $s,t$.
We now impose the restriction that $t + s > 0, s > 0$ and therefore write \eqref{IDEF} as $I^+$.


x





\begin{lemma} \label{detLEM}
The phase  has  non-degenerate critical points 
$$t =0, r = s, \sigma_2 = \sigma_1 = \frac{1}{2 s}, $$
 for $s\not=0$ in the variables $t,r, \sigma_1, \sigma_2$.  The determinant and signature of the  Hessian are given by,
\begin{itemize}
\item 
 $\det \wt \hcal_2 = (2s)^4$, hence,  $\det^{-\half} = (2s)^{-2}$. 
 \item $\rm{sgn} \wt \hcal_2 =0$
 \end{itemize}
\end{lemma}
The calculations are straightforward; since our only purpose is to calculate the amplitude, in the interest of brevity, we leave the calculations
in Lemma \ref{detLEM}  to the reader.

We then apply stationary phase to \eqref{IDEFSP}  in the variables $(t, \sigma_1, r, \sigma_2)$ to obtain the
complete asymptotic expansion claimed in the Proposition.  Using Lemma \ref{detLEM},  and cancelling the determinant factor of $s^{-2}$ with the Hadamard parametrix factors $(s + t) s |_{t=0}$, we obtain for $s \not= 0$, and for a dimensional constant $C_{n,d}$, the leading term in the asymptotic expansion has the form, for $s > 0$,
\begin{equation} \label{IDEFSP}\begin{array}{l}
I^+_{\rho}(q, s, \omega, \lambda) \simeq C_{n,d}^+ \lambda^{-2} \hat{\rho}(0) s^{d-1} 
s^{- \frac{n-1}{2}}   s^{-\frac{d-1}{2}} 
\wt{A}(s,  \frac{\lambda }{2s}, \frac{\lambda}{2s}). 
\end{array}  \end{equation}
We further integrate over $S^* H$ to obtain the final result.  Taking
into account the regularization of $\sigma_+^s$, the leading term with $s + t \geq 0, s \geq 0$  takes the form,
\begin{equation} \label{N1SP2} \begin{array}{l}
N^{1 +}_{\psi, \rho, H}(\lambda) \simeq    C_{n,d}^+ \lambda^{\frac{n +d}{2}+1}  \lambda^{-2} \hat{\rho}(0) \int_{\R} \hat{\psi} (s) s^{d-1} 
s_+^{- \frac{n-1}{2}}   s_+^{-\frac{d-1}{2}}   \left(  \int_{S^* H} 
\wt{A}( s,  \frac{\lambda }{2s}, \frac{\lambda}{2s}) dV_H(q) d S(\omega) \right)ds,  \end{array} \end{equation}
or more precisely, replacing $\wt{A}$ by its principal term \eqref{A0},
$$
\lambda^{\frac{n +d}{2}-1} C^+_{n,d}  \hat{\rho}(0) \int_{\R} \hat{\psi} (s) 
s_+^{- \frac{n-d}{2}}  \left(  \int_{S^* H} 
\Theta^{-\half}_M(q, \exp_q s \omega) \Theta^{-\half}_H(q, \exp_q s \omega)   dV_H(q) d S(\omega) \right) ds.  $$

\end{proof}


The stationary phase expansion for $t + s < 0, s < 0$ is the complex conjugate of that for $t + s > 0, s> 0$, as one 
sees by changing varibles  $s = - S, t = - T$ with $S, T > 0$. Hence the amplitude is the same. 

There are no critical points when  $t + s > 0, s < 0$, resp.
 $t + s < 0, s > 0$. 
This is because, by Proposition \ref{M0PROP}, the only $(s,t)$ for which there exist points in the canonical
relation contributing to the singularities of $S(t, \psi)$ are those 
for which there exist solutions of  \eqref{GEO1}, i.e. for which there exist $\xi \in T^*H$ such that 
$  G_H^{-s} \pi_H G_M^{t + s} (q, \xi) = \pi_H(q, \xi) $. This forces $t=0$ and then $s > 0$ and $s < 0$ are incompatible.

Thus, we have proved that the amplitude in \eqref{acHpsi1}  is as claimed in Theorem \ref{main 1}.

\section{\label{LONGtSINGS} Singularities of $S(t, \psi)$ for long times}

To prove the last statement of Theorem \ref{main 1} we will need a generalization of Theorem \ref{main 2}   on the asymptotics of $N_{\psi, \rho,H}^1(\lambda)$ to the case where $\rm{supp} \hat{\rho}$ is an arbitrarily long interval. We  assume as before 
that the solution set of \eqref{GEO1} is clean. The statement and proof are analogous to \cite[Proposition 1.20]{WXZ21},
and only involve the wave front analysis in Section \ref{WFSECT}. We only sketch the main points and refer to the
discussion of the $c < 1$ case in \cite{WXZ21} for further details.

\begin{proposition} \label{MORESINGS} 
  Let $\rho \in \scal(\R)$ with $\hat{\rho} \in C_0^{\infty}$ and with $0 \notin {\rm supp} \hat{\rho}$. Assume that  the fixed point set of $G_H^S$ at a period $S \in \Sigma^1$  is clean, and denote by  $d_j$ the dimension  of a component $Z_j(T)$ of the fixed point set.   Then, there exists $\beta_j \in \R$ and a complete asymptotic expansion,
$$  N^{c} _{\rho, \psi, H  }(\lambda) \sim \lambda^{-1 + \half (n -d) }  \sum_{T \in \Sigma^1} \sum_{\ell=0}^{\infty} \beta_{\ell} (t -T) \; \lambda^{\frac{d_j(T) }{2} -\ell},$$
The asymptotics corresponding to $T$ are  of lower order than the principal term of Theorem \ref{main 2}  unless $G_H^T = id$.\end{proposition} 

\begin{proof} We follow \cite{DG75, WXZ21}. 
The singularities of $S(t, \psi)$ are isolated and Lagrangian and we treat them one at a time. For $t$ sufficiently close to $T$,  $$S(t, \psi) = 
\sum_j \beta_j(t - T), $$ where $\beta_j$ is a homogeneous Lagrangian distribution given by,
$$
\beta_j(t) = \int_{\R} \alpha_j(s) e^{- i s t} ds, \;\; {\rm with}\;\; \alpha_j(s) \sim (\frac{s}{2 \pi i})^{ -1 + \half (n -d)+\frac{d_j(T)}{2}}\;\; i^{- \sigma_j} \sum_{k=0}^{\infty} \alpha_{j,k} s^{-k},$$
where $d_j(T) $ is the dimension of the component $Z_j(T)$.  
\end{proof}

\section{Tauberian theorems and proofs of Theorem \ref{main 1} and Corollary \ref{theo JUMP}}In the next section, we  use the following Tauberian theorems. 
Let  $\rho$ be a nonnegative Schwartz-class function on $\R$ with compact Fourier support let  $N$  be a tempered, monotone non-decreasing function with $N(\lambda) = 0$ for $\lambda < 0$, and $N'$ its distributional derivative as a nonnegative measure on $\R$.

\begin{proposition}[Corollary B.2.2 in \cite{SV}] \label{tauberian 1} Let $\rho \in \scal(\R)$ be a positive, even  
test function with $\hat{\rho}(0) = 1$ and $\hat{\rho} \in C_0^{\infty}(\R)$. Let $N(\lambda)$ be a monotone non-decreasing temperate function. 
	Fix $\nu \geq 0$. If $N' * \rho(\lambda) = O(\lambda^\nu)$, then
	\[
		N(\lambda) = (N * \rho)(\lambda) + O(\lambda^\nu).
	\]
	This estimate holds uniformly for a set of such $N$ provided $N' * \rho(\lambda) = O(\lambda^\nu)$ holds uniformly.
\end{proposition}

The next one is  \cite[Theorem B.5.1]{SV} with  $\nu = \frac{n+d}{2} - 1$.

\begin{proposition}\label{tauberian 2}
Let $\rho \in \scal(\R)$ be a positive, even  
test function with $\hat{\rho}(0) = 1$ and $\hat{\rho} \in C_0^{\infty}(\R)$. Let $N(\lambda)$ be a monotone non-decreasing temperate function.  If $N' * \rho(\lambda) = O(\lambda^{\frac{n+d}{2} - 1})$ and additionally
	\[
		N' * \chi(\lambda) = o(\lambda^{\frac{n+d}{2} - 1})
	\]
	for every Schwartz-class $\chi$ on $\R$ whose Fourier support is contained in a compact subset of $(0,\infty)$. Then,
	\[
		N(\lambda) = N * \rho(\lambda) + o(\lambda^{\frac{n+d}{2} - 1}).
	\]
\end{proposition}



Last, we recall \cite[Theorem 29.1.5-Corollary 29.1.6]{HoIV} in a form stated by Ivrii \cite{I80}. \begin{proposition}\label{tauberian 4} Let $\beta \in C_0^{\infty}(\R), \beta \equiv 1$ in $(-\half, \half)$, $\beta = 0$ for $|t| \geq 1$. Let
$\beta_T(t) = \beta(t/T)$. Let $N(\lambda)$ be a non-decreasing function such that $N(\lambda) \leq C \lambda^d$. 
Suppose that
$$\int_0^{\infty} \hat{\beta}_T(\lambda - \mu) d N(\mu) = a_0 d \lambda^{d-1} + a_1(d-1) \lambda^{d-2} +
o(\lambda^{d-2}). $$
Then,
$$|N(\lambda) - a_0 \lambda^d + a_1 \lambda^{d-1} | \leq C \frac{a_0}{T} \lambda^{d-1} + o(\lambda^{d-1}). $$
\end{proposition}

\subsection{ \label{SS Tauberian} Completion of the proof of Theorem \ref{main 1} }

Except for the last statement, 
Theorem \ref{main 1} follows from Theorem \ref{main 2} and a standard cosine Tauberian theorem:
Except for the last statement on aperiodic manifolds, the remainder of the proof of Theorem \ref{main 1} is similar to the end of the proof of Theorems 1.16 and 1.20 
of \cite{WXZ21}. We therefore sketch the overlapping proofs and refer to the earlier paper for complete details.

\begin{proof} 
Theorem \ref{main 1} pertains to the Weyl function $N^{1} _{\psi, H  }(\lambda)$ of  \eqref{cpsi}, which for convenience we repeat here,
$$N^{1} _{\psi, H  }(\lambda): = 
\sum_{j, k:  \lambda_k \leq \lambda}  \psi( \lambda_j - \mu_k)  \left| \int_{H} \phi_j \overline{e_k}dV_H \right|^2. $$
We assume with no essential loss of generality that $\psi \geq 0$. Then, $N^{c} _{\psi, H  }(\lambda)$ 
 is monotone non-decreasing and has Fourier transform $S(t, \psi)$ \eqref{St}.\eqref{SpsiDEF}

For $\psi \geq 0$, we apply Proposition \ref{tauberian 1} with $\hat{\rho}\; \cap\; {\rm singsupp}\; S(t, \psi) = \{0\}$ and  to $d N^{1} _{\psi, H  }(\lambda) $. By  Theorem \ref{main 2},  $\rho* d N^{1} _{\psi, H  }(\lambda)  = \beta_0 \;\lambda^{\frac{n+d}{2} -1} + O(\lambda^{\frac{n+d}{2} -2})$, and therefore,
$$\begin{array}{lll}  N^{1} _{\psi, H  }(\lambda) & = &  \rho* N^{1} _{\psi, H  }(\lambda) + O(\lambda^{\frac{n+d}{2} -2} ) \\ &&\\
& = & \beta_0 \lambda^{\frac{n+d}{2}} +  O(\lambda^{\frac{n+d}{2} -1} ),\end{array}  $$
where $\beta_0$ is the principal coefficient,
concluding the proof of Theorem \ref{main 1}.
\end{proof}

\subsection{\label{Aperiodic} Aperiodic case: Proof of the last statement of  Theorem \ref{main 1}. }

It remains to prove the last statement of Theorem \ref{main 1}, that 
if the geodesic flow $G^t_H$ of $H$ is aperiodic, then $$N_{\psi, H}^{1} (\lambda) =
C_{n,d} \;\;   a^0_1(H, \psi) \; \lambda^{\frac{n+d}{2}   } + R_{\psi, H}^{1} (\lambda), \;\; \rm{where}\; R_{\psi, H}^{1} (\lambda) =   o(\lambda^{\frac{n+d}{2} -1}).$$

In the case where the fixed point sets of $G_H^S$ at a period $S \in \Sigma^1$  are all clean, the last statement
follows immediately from Proposition \ref{MORESINGS} and Proposition \ref{tauberian 2}. Hence, the problem is to prove the same estimates
without assuming cleanliness.
 We only assume that the closed geodesics at all non-zero periods
forms a set of Liouville measure zero. 

\begin{rem} In principle, there could exist a second term of order $\lambda^{\frac{n+d}{2} -1}$. It requires a calculation 
to prove that it vanishes in the Kuznecov $c=1$ case (Section \ref{SUBPRINCIPAL}), as it did in the case $c < 1$ \cite{WXZ21}. \end{rem}

\begin{proof} In Theorem \ref{main 1}-Theorem \ref{main 2}, we have
already proved an  asymptotic result  when $\hat{\rho}$ only contains
$\{0\}$ among the singularities. 
To obtain the two-term Weyl law for longer times, we use a pseudo-differential cutoff argument generalizing the one for pointwise
Weyl asymptotics in   \cite[Theorem 29.1.5-Corollary 29.1.6]{HoIV}.

Let $\hat{B}_T, \hat{b}_T: =I - \hat{b}_T  \in \Psi^0(H)$ be zeroth order pseudo-differential operators on $H$  so that the support of the principal symbol  $b_T$ of $\hat{b}_T$  contains the union of
all closed geodesics of $H$ of period $\leq T$.  Let us briefly review
the construction of $B_T, b_T$ from \cite{HoIV}. First define the microlocal period function of $H$, \begin{equation}\label{M1}
L_H^*(q,\eta)=\inf \{t>0: \,G^t_H(q, \eta) =(q, \eta)  \},
\end{equation}
where $L^*$ is defined to be $+\infty$ if no such $t$ exists.  It is  homogeneous
of degree zero and lower semicontinuous.  Henceforth, we restrict it to $S^*H$.  The  set of periodic points of $G^t_H$ is the closed set defined by  \begin{equation} \pcal_H  =  \{(q,\eta) \in S^*H: \, 1/L_H^*(q,\eta)\ne 0\}.  \end{equation}
 If $T>1$
is a large parameter, then we can find a function $b_T\in C^\infty (S^*H, [0,1])$ so that
\begin{equation} \label{W4} \int_{S^*H }b_T(q, \eta) d \mu_L(q, \eta) \le 1/T^2, \end{equation}
and so that
$$ 1/L_H^*(q,\eta)\le 1/T, \quad \text{on} \, \, \text{supp } B_T\; (= \text{supp} (1 -b_T)).$$


 We then define $\hat{B}_T = Op_H(B_T)$ for a fixed choice
of quantization; similarly for $\hat{b}_T$ and
 use the partition of unity $I = \hat{B}_T + \hat{b}_T$ to introduce pseudo-differential  cutoffs on $L^2(H)$  to decompose the trace \eqref{S(t, psi)DEF}. There are several ways to introduce  cutoffs in the composition $\gamma_H U_M\gamma_H^* (t  + s, q, q') U_H(-s, q, q') $: (i) to introduce  $I = \hat{B}_T + \hat{b}_T$ only on the left and right  sides of 
$U_H(-s) $ (with adjoint on the right side) or only on the left and right sides of $\gamma_H U_M(t +s) \gamma_H^*$; (ii)  to introduce the partition of unity on both sides of both factors. 
It turns out that (ii) is a convenient choice in apply Proposition \ref{tauberian 2}. In terms of eigenfunction 
expansions, it corresponds to 
\begin{equation} \label{NpsirhoDEFB} N^{1} _{\psi, \rho, H  }(\lambda) = \sum_{j, k=0}^{\infty} \rho(\lambda - \lambda_j)  \psi(\lambda_j - \mu_k)   \left| \langle (B_T + b_T) \gamma_H \phi_j),\overline{(B_T + b_T)e_k} \rangle_{H} \right|^2, \end{equation}
where $\langle f, g \rangle_H = \int_H f \bar{g} dV_H$. The goal is to use Proposition \ref{tauberian 2} to show
that \eqref{NpsirhoDEFB}  is $o(\lambda^{\frac{n+d}{2} -2})$.

A crude but effective approach is to multiply out the inner product and the modulus-square and estimate each 
resulting term separately. 
 Multiplying out the inner product 
$\langle (B + b) \gamma_H \phi_j),\overline{(B + b)e_k} \rangle_{H} $ we obtain four terms. We call the ones
with $(B,B)$, resp.  $(b,b)$ `diagonal terms'  $D_1$ resp. $D_2$ and the ones with mixed $(B,b)$ resp. $(b,B)$,  `off-diagonal terms'
$A_1, $ resp. $A_2$. Then multiplying out $|D_1 + D_2 + A_1 + A_2|^2$ gives the `pure' products $|D_1|^2 + |D_2|^2 + |A_1|^2 + |A_2|^2$
plus the `mixed' products, of which one is  $\Re D_1 \bar{D}_2$, four are of the form $\Re D_k \bar{A}_j$ and one is $\Re A_1 \bar{A}_2$. Using $|\Re a \bar{b} | \leq \half (|a|^2 + |b|^2)$ we can bound all of the mixed products by a universal
constant times the pure products. It therefore suffices to show that the analogue of \eqref{NpsirhoDEFB} 
with summand  $|D_1|^2 + |D_2|^2 + |A_1|^2 + |A_2|^2$ is $o(\lambda^{\frac{n+d}{2} -2})$.
We denote the corresponding sums by,
\begin{equation}\label{FOUR}
N^1_{pure, \psi, \rho, H}(\lambda) = N^1_{|D_1|^2,\psi, \rho, H}(\lambda) + N^1_{|D_2|^2, \psi, \rho, H}(\lambda)
+  N^1_{|A_1|^2\psi, \rho, H}(\lambda) +  N^1_{|A_2|^2,\psi, \rho, H}(\lambda). \end{equation}
Each term is a monotone  non-decreasing temperate function in the sense of hypotheses of Proposition \ref{tauberian 2}.

As mentioned in Section \ref{SIMPLE}, the  asymptotics of the four terms of \eqref{FOUR} can be determined by the same method as for \eqref{NpsirhoDEF}. We express each term as a semi-classical Fourier transform
$\fcal_{t \to \lambda}$ of the corresponding part  (e.g. $S_{|D_1|^2} (t, \psi)$) of of the Kuznecov trace, which we express in terms  of the wave kernels composed with the designated
pseudo-differential operators. The wave front analysis of these four operator traces is the same as in Section \ref{WFSECT}, the only change being
in the formulae for the amplitudes and symbols. 


 
\subsubsection{The $(B,B)$ term}
By  the assumption on $B$ and the wave
front analysis in Section \ref{WFSECT},  the kernel  $$K_{BB} (s, t, q,q') 
=B_T \gamma_H U(t +s ) \gamma_H^* B_T^* \circ B_T U_H(-s)B_T^*(q,q')$$ is smooth for $0<|t|<T$.  Indeed, 
there do not exist any solutions of  \eqref{GEO1} for $t \not= 0$ in the support of $B_T$ and therefore the only 
solutions are those of \eqref{GEO2}. 
It follows that, as long as $\rm{Supp} \hat{\chi} \subset (-T, T)$,  $ \chi * d  N^1_{|D_1|^2,\psi, \rho, H}(\lambda)   $  has a complete asymptotic expansion as in 
Theorem \ref{main 2} and Proposition \ref{MORESINGS}. 
To employ Theorem \ref{tauberian 2}, we start with a given $\chi \in \scal(\R)$ with $\hat{\chi} $ vanishing near $0$
and supported in $[-T, T]$ and then decompose $1 = B_T + b_T$. Then $S_{BB}(t, \psi)$ is smooth for
$t \in \rm{Supp} \hat{\chi}$. Hence,  $ \chi * d  N^1_{|D_1|^2,\psi, \rho, H}(\lambda) = O(\lambda^{-\infty})  $.

\subsubsection{The $(b,b)$ term}

With $\chi, T$ fixed as above, we next consider   $ \chi * d  N^1_{|D_2|^2,\psi, \rho, H}(\lambda) $. To prove
that this monotone function is $o(\lambda^{\frac{n+d}{2} - 2})$ we use Proposition \ref{tauberian 3} to prove that for any $\epsilon > 0$,  \begin{equation} \label{ADDCONDEP} 
	 \chi * d N_{\psi, H}^1 (\lambda) \leq C \; \epsilon\; \lambda^{\frac{n+d}{2}-1}.
\end{equation}
Indeed, this estimate holds if $\chi$ is replaced by $\rho$ with $\hat{\rho} $ supported in $(-r_0, r_0)$. As mentioned
above, the proofs of the  asymptotic expansion of Theorem \ref{main 2} and of the Kuznecov-Weyl law Theorem \ref{main 1} extend with only minor modifications if we compose with $b(x, D)$. The modification is that the integrand of the principal
term $a^1_0(H, \psi)$ acquires the additional factor of the principal symbol $b_0$ of $b(x, D)$ and is therefore
of order $\frac{1}{T} < \epsilon$. By Proposition \ref{tauberian 3} the same estimate holds general $\chi$.

\end{proof}

 \subsection{$(b,B)$ terms} For such `off-diagonal terms, we move both cutoffs onto the $e_k$ factor of the inner products, as in 
  $$\left| \langle  \gamma_H \phi_j,(b_T)^*\overline{(B_T)e_k} \rangle_{H} \right|^2. $$
As in the case of $(B,B)$ terms, $(b_T)^* (B_T) U_H(s)$ has a smooth kernel for $s \in (0, T)$. Hence, the contributions
of these term is the same as their cutoff to a small interval around $s=0$ multiplied by $\frac{1}{T}$. 

\subsection{Conclusion} It follows that if the expansion of $N^{1}_{\psi, \rho, H}(\lambda)$ is the same as the
expansion in the case where $\rm{Supp} \hat{\rho} \subset (-r_0, r_0)$ plus $\epsilon \lambda^{\frac{n+d}{2}-1}$
plus $\frac{1}{T} \lambda^{\frac{n+d}{2}-1}$. Hence, assuming that the second term in the expansion at $t=0$ vanishes,
$$N^{1}_{\psi, \rho, H}(\lambda) = a_0^1(H, \psi) \lambda^{\frac{n+d}{2}-1} + \epsilon \lambda^{\frac{n+d}{2}-1}$$
for any $\epsilon > 0$, proving the last statement of Theorem \ref{main 1}.
 
\subsection{\label{PFTHEOJUMP} Proof of Corollary  \ref{theo JUMP} } 

\begin{proof}

Theorem \ref{theo JUMP} is a consequence of the remainder estimate of Theorem \ref{main 1}.
To prove Theorem  \ref{theo JUMP} it suffices to prove that, for any $\epsilon > 0$ there exists a test
function $\psi \geq 0, \hat{\psi} \in C_0^{\infty}(\R), \hat{\psi}(0) =1$ with $\rm{Supp} \hat{\psi} \subset (-r_0, r_0)$ and  $\psi \geq \mathbf{1}_{[-\epsilon,\epsilon]}$. Then there exists   a universal constant $C(\epsilon)$
depending only on $(\epsilon, \delta)$  so that for all $\lambda_j$,
\begin{equation} \label{LB} J_{\psi, H}^1(\lambda_j) \geq C(\epsilon) \; J_{\epsilon, H}^{1} (\lambda_j). \end{equation}
Then, 
\[
	\sum_{ k: |\mu_k - c \lambda_j | \leq \epsilon  }    \left| \int_{H} \phi_j \overline{e_k}dV_H \right|^2 \leq \sum_k  \psi(\lambda_j -\mu_k)\left| \int_{H} \phi_j \overline{e_k}dV_H \right|^2,
\] and the upper bound for $ J_{\psi, H}^1(\lambda_j) $ given  in Corollary \ref{theo JUMP} provides the
upper bound for $ J_{\epsilon, H}^{1} (\lambda_j)$. 

The construction of $\psi = \psi_{\epsilon}$ is elementary and we follow the discussion in \cite[Lemma 2.3]{DG75}. 
Let $\hat{\psi} \in \scal(\R)$ with $\hat{\psi} \in C_0^{\infty}$. Replacing $\psi$ by $\psi \cdot \bar{\psi}$ one has $\psi \geq 0, 
\psi(0) > 0$ and $\hat{\psi} \in C_0^{\infty}$. Replacing $\hat{\psi}$ by $\hat{\psi}(\frac{s}{\delta})$ and $\psi$ by $\delta \psi(x \delta)$, and
taking $\delta$ sufficiently small one can assume that $\rm{supp} \hat{\psi} \subset (-r_0, r_0)$ and by multiplying by a
positive scalar we have,  $\psi > 0$ on $[- K, K]$ for any $K > 0$.

\end{proof}

In the aperiodic case, the same argument gives \eqref{JUMPAP}. 

\subsection{\label{SUBPRINCIPAL} Subprincipal term} 

The vanishing of the subprincipal term is a result pertaining to the smooth expansions in Theorem \ref{MORESINGS} and is independent of 
the Tauberian argument. As in \cite{WXZ21},  it follows from the fact that the subprincipal symbol of $\sqrt{\Delta_X}$ vanishes
for any Riemannian manifold $X$, together with some parity arguments from \cite{DG75}. We assume that 
 $\hat{\psi}$ has small support and both $\hat{\psi}$
and $\hat{\rho} $ are even, and that $\hat{\rho} \equiv 1$ near $0$.

We claim that     the subprincipal symbol is  odd, so that  its integrals over cospheres vanishes.
We first note that     the subprincipal symbols of $\sqrt{-\Delta_M} \otimes I$ and of $Q_c$ both vanish.   The homogeneous
part of degree k in  $\sigma_P(x, \xi)$ is even, resp. odd if k is even, resp. odd. By induction with respect to  $r$ it follows that $(\frac{\partial}{\partial t})^r
a_{-j}$ is an even, resp. odd. if $r-j$ is even, resp. odd. 
    The amplitude of
    $e^{it P} e^{i s Q_c}$    is obtained by integrating  $e^{i (t - c s) P_M} \otimes     e^{i s P_H} $. The parities of the terms in the amplitude agree with those of \cite{DG75},
for $s = t = 0$. The restriction of the $M$-amplitudes to $H$ have the same parity. The further restriction  to
    the diagonals in $H \times H$  seems to multiply the amplitudes, but  the subprincipal term can only be obtained as the product of the
    principal symbol and the  subprincipal symbol. Hence it is odd.



  \section{Appendix}

\subsection{\label{BDSECT} Blow-down singularity}

In this section, we  review the definition of a blowdown map  $f: \R^N \to \R^N$, following  \cite[Page 111]{G89}.
$f: X^n \to Y^n$ is a blow-down map with singularity along  a submaifold  $S$ if 
\begin{itemize}
\item $S =  \{x: \det D f(x) = 0\}$ is the critical set of $f$. One assumes that
$d (\det D f(x)) \not= 0\;\; \rm{on} \;\; S$, so that $S$ 
is a  smooth submanifold. 
\item $\ker d_s f \subset T_s S$ for all $s \in S$;
\item $f$ is of constant rank along $S$ and $f^* dV_x$ vanishes to order $n-k$ along $S$.

\end{itemize}

Roughly,  $f |_S :S \to W$ is a fibration over a submanifold $W \subset Y$ of codimension $n- k +1$.
Under these assumptions,  there exist coordinates $x_1, \dots, x_n$ around each $s \in S$ and $y_1, \dots, y_n$ around
$f(s)$ in $Y$ so that $f^* y_j = x_j$ for $j=1, \dots, k$ and $f^* y_i= x_i x_1$ for $i = k +1, \dots, n$.
In this case, $S = \{x_1 = 0\}$ and $W = \{\vec y: y_1 = y_{k+1} = \cdots y_n =0\}. $

We claim that the Lagrange map $\iota_{\Psi}$ of Lemma \ref{LAGLEM} and of the model phase \eqref{PSIMODEL} is a blow-down map. The
 critical point set of \eqref{PSIMODEL} is given by,
$$\begin{array}{ll} \nabla_{y_1, \dots, y_{d-1}}  \Psi_{model} = \vec x \in \R^{d-1},  & \;\;
\nabla_{y_d} \Psi_{model} = -\half (x_1^2 + \cdots + x_{d-1}^2 +
x_{d+1}^2 + \dots + x_n^2) ) = 0, \\ & \\
\nabla_{(x_1, \dots, x_{d-1})} \Psi_{model} = \vec y \in \R^{d-1},  & \; 
\nabla_{(x_{d +1}, \dots, x_n)}\Psi_{model}  = y_d \vec x \in \R^{n-d}. \end{array} $$
The phase variables are $(y_1 \dots, y_{d-1}, x_1, \dots, x_{d-1})$. 
The second equation forces $x_d =1$. 

In the definition of $f$, we let  $$X = \{ (y_d, x_{d+1}, \dots, x_n)\} \simeq \R^{n-d}, \;\;
Y= \{ ((y_1, \dots, y_{d-1}, x_1, \dots, x_{d-1})\} \simeq \R^{2(d-1)}, $$
and define the critical set of the phase by,
\begin{equation} \label{CphiMOD} C_{\Psi_{model}} =\{ (y_1, \dots, y_d; x_1, \dots, x_n) : (x_1, \dots, x_{d-1}) = 0 =  (y_1, \dots, y_{d-1})\}
\subset X \times \R^{2(d-1)}. \end{equation}
Then the associated Lagrange map $\iota_{\Psi_{model}} : C_{\Psi_{model}} \to T^* \R^{n-d}$ is given by,
\begin{equation} \label{iotaPhiMOD} \begin{array}{lll}\\\ &&\\
\;\; \iota_{\Psi_{model}}(\vec 0, y_d, \vec 0, x_{d+1} \dots, x_n) & = &  (\vec 0, y_d, d_{\vec y} \Psi_{model}, \vec 0, (x_{d+1}, \dots, x_n), d_{x''} \Psi_{model}) \\ &&\\ 
& = & (\vec 0, y_d, 0, \vec 0, (x_{d+1}, \dots, x_n), y_d(x_{d+1}, \dots, x_n) ). \end{array} \end{equation}
The image is a (non-homogeneous) Lagrangian submanifold of $T^* \R^{n-d}$, in which the fiber $(x_{d+1}, \dots, x_n)
\in \R^{n-d-1}$
gets blown down to a point when $y_d =0$. In the original model over $T^* H$ and with $\wt \omega = e_d$, the 
set $y_d $ corresponds to the diagonal, and $S^* H$ gets blown down to a point when $s =0$.

\end{document}